\documentclass[final]{siamart190516}

\usepackage[T1]{fontenc}
\usepackage{amssymb}
\usepackage{amsmath}
\usepackage{mathrsfs}
\oddsidemargin=\evensidemargin

\newtheorem{remark}[theorem]{Remark}
\newtheorem{notation}[theorem]{Notation}

\numberwithin{equation}{section} 

\newcommand{\grad}{\operatorname{\mathbf{grad}}}
\newcommand{\Rot}{\operatorname{\mathbf{curl}}}
\newcommand{\Div}{\operatorname{\mathrm{div}}}
\newcommand{\rot}{\operatorname{\mathrm{curl}}}

\newcommand{\paren}[1]{\left(#1\right)}

\newcommand{\loc}{\mathrm{loc}}
\newcommand{\tang}{\mathrm{t}}
\renewcommand{\Im}{\operatorname{Im}}
\renewcommand{\Re}{\operatorname{Re}}

\newcommand{\kappai}{\kappa_{\rm i}} 
\newcommand{\kappae}{\kappa_{\rm e}}
\newcommand{\mui}{\mu_{\rm i}}
\newcommand{\mue}{\mu_{\rm e}}

\newcommand{\xb}{\mathbf{x}}
\newcommand{\yb}{\mathbf{y}}
\newcommand{\nhat}{\widehat{\mathbf{\n}}}
\newcommand{\xhat}{\widehat{\mathbf{x}}}
\newcommand{\yhat}{\widehat{\mathbf{y}}}
\newcommand{\zhat}{\widehat{\mathbf{z}}}
\newcommand  {\C}{{\mathbb C}}
\newcommand  {\N}{{\mathbb N}} 
\newcommand  {\R}{{\mathbb R}}
\renewcommand  {\S}{\mathbb{S}}
\renewcommand  {\P}{\mathbb{P}}
\renewcommand  {\H}{\mathbb{H}}
\newcommand {\Id}{\mathrm {I}}

\newcommand {\Omegae}{\Omega^{\rm c}}
\newcommand {\Gammaref}{\Gamma_{{\hspace{-.3mm}\rm ref}}}
 
\newcommand  {\D}{\boldsymbol{\operatorname{D}}}
\renewcommand  {\L}{\boldsymbol{\mathcal{L}}}
 
\newcommand  {\LL}{\boldsymbol{ L}}
\newcommand  {\HH}{\boldsymbol{ H}}
\newcommand  {\TT}{\boldsymbol{ T}}
\newcommand  {\EE}{\boldsymbol{ E}}
\newcommand{\FF}{F}
\newcommand{\F}{F} 
\newcommand{\FFop}{G}
\newcommand{\dFFop}{\mathbf{\FFop}}
\newcommand{\pFFop}{\mathcal{\FFop}}
\newcommand  {\Y}{\boldsymbol{\mathcal Y}} 
\newcommand  {\parset}{{\mathcal Q}}

\newcommand{\Pq}{\mathcal{P}_{\qq}}

\newcommand{\hh}{{\boldsymbol{ h}}}
\newcommand  {\nn}{\boldsymbol{ n}}
\newcommand  {\uu}{\boldsymbol{ u}}
\newcommand  {\jj}{\boldsymbol{ j}}
\newcommand  {\mm}{\boldsymbol{ m}}
\newcommand  {\vv}{\boldsymbol{ v}}
\newcommand  {\ww}{\boldsymbol{ w}} 
 \newcommand  {\ee}{\boldsymbol{ e}}
 \newcommand  {\ff}{\boldsymbol{ f}}
 \newcommand  {\s}{\boldsymbol{s}}
 \newcommand  {\js}{\jj_{\hspace{-.5mm}\s}}
 \newcommand  {\ms}{\mm_{\rm s}}
 \renewcommand  {\gg}{\boldsymbol{ g}}
  \newcommand  {\qq}{\boldsymbol{ q}}
    \newcommand  {\dd}{\boldsymbol{ d}}
    \newcommand  {\ki}{\boldsymbol{ \xi}}
      \newcommand  {\pp}{\boldsymbol{ p}}
  \renewcommand  {\tt}{\boldsymbol{t}}
 \newcommand  {\n}{{\boldsymbol{\eta}}}
 \newcommand  {\fq}{\ff_{\hspace{-.5mm}\qq}}

\newcommand{\FS}[2]{\Phi(#1,#2)}

\newcommand{\smat}[1]{\left(\!\begin{smallmatrix} #1 \end{smallmatrix}\!\right)}

  \newcommand{\transposee}[1]{{\vphantom{#1}}{#1}^{\sf T}}  
       \newcommand{\adjoint}[1]{{\vphantom{#1}}{#1}^{\text{\small\hspace{-2mm}\begin{tabular}{c}\vspace{-1.5mm}\footnotesize{*}\end{tabular}}\hspace{-2mm}}} 
       \newcommand{\Adjoint}[1]{{\vphantom{#1}}{#1}^{\text{\small\hspace{-2.5mm}\begin{tabular}{c}\vspace{-1.5mm}\footnotesize{*}\end{tabular}}\hspace{-2mm}}}   

\begin{document}

\renewcommand{\thefootnote}{\fnsymbol{footnote}}
\footnotetext{Institut f\"ur Numerische und Angewandte Mathematik,
Georg-August Universit\"at G\"ottingen,
37083 G\"ottingen, Germany, E-mail:
hohage@math.uni-goettingen.de}
\renewcommand{\thefootnote}{\arabic{footnote}}

\title{A spectrally  accurate method for the dielectric obstacle scattering problem and
applications to the inverse problem}
\author{Thorsten Hohage \and Fr\'ed\'erique Le Lou\"er}
\maketitle

\begin{abstract} We analyze the inverse problem to reconstruct the shape of a 
three dimensional  homogeneous  dielectric obstacle from  the knowledge of 
noisy far field data.  The forward problem is solved by 
a  system of second kind boundary integral equations.   
For the numerical solution of these coupled integral equations we propose a 
fast spectral algorithm by transporting these equations onto the unit sphere. 
We review the differentiability properties of the boundary to 
far field operator and give a characterization of the adjoint operator of the 
first  Fr\'echet derivative. Using these results we discuss the implementation
of the iteratively regularized Gauss-Newton method for the numerical
solution of the inverse problem and give numerical results
for star-shaped obstacles. 
\end{abstract}


\section{Introduction}
The problem to reconstruct the shape of scatterers from noisy far field measurements of time-harmonic  waves arises in many fields of applied physics, as for example sonar and radar applications, bio-medical imaging and non destructive testing.  Such inverse problems are severely ill-posed. 
Often they are formulated as a nonlinear least squares problem, for which regularized iterative algorithm can be applied to recover an approximate solution.

The numerical treatment of the inverse problem requires  the investigation of the forward problem.  Here we consider the scattering of time-harmonic waves at a fixed frequency $\omega$ by a three-dimensional bounded and  non-conducting homogeneous dielectric  obstacle $\Omega$. The electric permittivity $\varepsilon$ and the magnetic permeability $\mu$ are assumed to take  constant values in the interior and in the exterior of the obstacle, but discontinuous across the interface $\Gamma$. The wavenumber is given by the formula $\kappa=\omega\sqrt{\varepsilon\mu}$. In this case the forward problem is described by the system of Maxwell equations in the whole space $\R^3$, with natural transmission conditions expressing the continuity of the tangential components of the magnetic and electric fields across $\Gamma$. Let $\Omegae$ denote the exterior domain $\R^3\backslash\overline{\Omega}$ and $\nn$ denote the outer unit normal vector on the boundary $\Gamma$. We  label the dielectric quantities related to the interior domain $\Omega$ by the index $\rm i$ and to the exterior domain $\Omegae$ by the index $\rm e$. Eliminating the magnetic field in the Maxwell system we obtain the  following transmission problem:  Given an incident electric wave $\EE^{\rm inc}$ which is assumed to solve the second order Maxwell equation in the absence of any dielectric scatterer, find the electric field solution $\EE=(\EE^{\rm i},\EE^{\rm s})$ that satisfies
\begin{subequations}
\begin{align}\label{ME} 
  \Rot\Rot \EE^{\rm s} - \kappae^2\EE^{\rm s} &=  0 &&\text{ in }\Omegae,\\ 
	\Rot\Rot \EE^{\rm i} - \kappai^2\EE^{\rm i} &= 0 &&\text{ in }\Omega,
\end{align}
and the transmission conditions on $\Gamma$, 
\begin{align}\label{T1}
\nn\times \EE^{\rm i}&=\nn\times(\EE^{\rm s}+\EE^{\rm inc}) \\
\tfrac{1}{\mui}\nn\times\Rot\EE^{\rm i}&=\tfrac{1}{\mue}\nn\times\Rot(\EE^{\rm s}+\EE^{\rm inc}).
\end{align} 
In addition the scattered field $\EE^{\rm s}$ has to satisfy the Silver-M\"uller radiation condition
\begin{equation}\label{T3}\lim_{|\xb|\rightarrow+\infty}|\xb|\left| \Rot \EE^{\rm s}(\xb)\times\frac{\xb}{| \xb |}- i\omega\mue\EE^{\rm s}(\xb) \right|
 =0\,.\end{equation}
\end{subequations}
uniformly in all directions  $\xb/ |\xb |$.

Well-posedness of  the dielectric obstacle scattering problem for any positive 
real values of the  dielectric constants is well known, and this 
 problem can be reduced in several different ways to 
 coupled or single  boundary integral equations on the dielectric  interface
$\Gamma$: 
For an overview of these formulations for smooth boundaries we refer to Harrington's book \cite{Harrington} and the paper of  Martin and  Ola \cite{MartinOla}. Some pairs of integral equations have irregular frequencies and others do not, as the so-called M\"uller's system \cite{Muller}. The occurrence of irregular frequencies can be avoided by the use of single combined-field integral equation method.  This idea was first suggested by Mautz \cite{Mautz}. Existence of the solution was then proved by Ola and Martin via a regularization technique. For a Lipschitz boundary,  Buffa, Hiptmair, von Petersdorff  and Schwab \cite{BuffaHiptmairPetersdorffSchwab} derived a uniquely solvable system of integral equations  and Costabel and Le Lou\"er \cite{CostabelLeLouer1,FLL} constructed a family of four alternative single boundary integral equations extending a technique due to Kleinman and Martin \cite{KleinmanMartin} in acoustic scattering.
 
 The radiation condition implies that the scattered field $\EE^{\rm s}$  has an asymptotic behavior of the form
$$\EE^{\rm s}(\xb)=\frac{e^{i\kappae|\xb|}}{|\xb|}\EE^{\infty}(\xhat)+O\left(\frac{1}{|\xb|}\right),\quad |\xb|\rightarrow\infty,$$
uniformly in all directions $\xhat=\tfrac{\xb}{|\xb|}$. The far field pattern $\EE^{\infty}$ is a tangential vector function defined on the unit sphere $\S^2$ of $\R^3$ and is always analytic.

\medskip
The forward problem discussed in this paper is the scattering of $m$ incident plane waves of the form 
$\EE^{\rm inc}_k(\xb)=\pp_k\,e^{i\kappae \xb\cdot \dd_k}$, $k=1,\dots,m$ where 
$\dd_k,\pp_k\in \S^2$ and $\dd_k\cdot\pp_k=0$. We denote by $\FF_k$ the boundary to far field operator that maps a parametrization of the boundary $\Gamma$ onto the far field pattern $\EE^{\infty}_k$ of the scattered field
$\EE^{\rm s}_k$ of the solution $(\EE^{\rm i}_k,\EE^{\rm s}_k)$ to the problem \eqref{ME}-\eqref{T3} for the 
incident wave $\EE^{\rm inc}_k$. For simplicity we do not distinguish between the boundary $\Gamma$ and its
parametrization in this introduction. The inverse problem consists in reconstructing $\Gamma$ given noisy 
measured data described by
\begin{equation}\label{IP}  
\EE^{\infty}_{k,\delta} = \FF_k(\Gamma)+{\bf err}_k\,,\quad k=1,\dots,m\qquad
\sum_{k=1}^m\|{\bf err}_k\|_{L^2}^2\leq \delta^2\,.
\end{equation}
Here measurement errors are described by the functions ${\bf err}_k$, and the error bound $\delta$,
the incident fields $\EE^{\rm inc}_k$, and the dielectric constants are assumed to be known. 
By straightforward modifications of the algorithms described in this paper, one could simultaneously
invert for $\Gamma$ and a dielectric constant. 
In this situation H\"ahner \cite{Hahner} has shown a uniqueness result 
assuming  knowledge of the far field patterns for all incoming plane waves. 
However, even with known dielectric constants it remains an open question whether or not $\Gamma$ is 
uniquely determined by only a finite number of incoming plane waves and known dielectric constants. 

Over the last two decades, much attention has been devoted to the investigation of efficient iterative method, in particular regularized  Newton-type  for nonlinear ill-posed problem via first order linearization  \cite{Bakushinskii, Hohage,Hohage2,Kirsch}. Until now, it was successfully applied to inverse acoustic scattering  problems \cite{Hohage, HohageHarbrecht}. Indeed, the use of such iterative  methods requires the analysis and an explicit form of the Fr\'echet derivatives of the boundary to far field operators $\FF_k$. The Fr\'echet derivative of the far field pattern is usually interpreted as a the far field pattern of a new scattering problem and  these characterizations are  well-known in acoustic scattering since the 90's. Many different approaches were used:  Frechet differentiability of the far field (or of the solution away from the boundary) was established by Kirsch \cite{Kirsch} and Hettlich \cite{Hettlich, HettlichErra} via variational methods, by Potthast via boundary integral representations \cite{ Potthast1, Potthast2}, by Hohage \cite{Hohage} and Schormann \cite{HohageSchormann} via the implicit function theorem and by Kress and Pa\"iv\"arinta via Green's theorem and a far field identity \cite{KressPaivarinta}. 

In electromagnetism, Fr\'echet differentiability was first investigated by Potthast \cite{Potthast3} for the perfect conductor problem extending the boundary integral equation approach. The characterization of the derivative was then improved by Kress \cite{Kress}. More recently, Fr\'echet differentiability  was analyzed by Haddar and Kress \cite{HaddarKress} for the Neumann-impedance type  obstacle scattering  problem via the use of a far field identity and by Costabel and Le Lou\"er \cite{CostabelLeLouer,CostabelLeLouer2,FLL} and Hettlich \cite{Hettlich2} for the dielectric scattering problem via the boundary integral equation approach \cite{CostabelLeLouer1, FLL} and variational methods, respectively. 

It is the purpose of the present paper to apply the iteratively regularized Gauss-Newton method to the inverse dielectric obstacle scattering problem. 
In section 2, we describe the two different  boundary integral equation methods that are used to solve the electromagnetic transmission problem via direct and indirect approaches. In section 3 we propose a new spectral method to solve these systems  which ensures superalgebraic convergence of the discrete solution to the exact solution, in the case of simply connected closed surface. Ganesh and Hawkins proposed first two methods, in the context of the perfect conductor problem,  by transforming the  integral equation on the surface $\Gamma$ in an integral equation on the unit sphere using a change of variable and then by looking for a solution  in terms of  series (component-wise) of scalar spherical harmonics \cite{GaneshHawkins1} or of series of vector spherical harmonics \cite{GaneshHawkins2}. To decrease the number of unknowns, they introduce in \cite{GaneshHawkins3} a normal transformation acting from the tangent plane to the boundary $\Gamma$ onto the tangent plane to the unit sphere, so that one only has to seek a solution in terms of tangential vector spherical harmonics. Here we use a different approach based on the Piola transform of a diffeomorphism from $\S^2$ to $\Gamma$ that maps  the energy space $\HH^{-1/2}_{\Div}(\Gamma)$ defined in the following section 
to the space  $\HH^{-1/2}_{\Div}(\S^2)$. The numerical implementation is discussed in section 4 and numerical results on the convergence rate of the method are presented.
In section 5 we recall the  main results on the  Fr\'echet differentiability of the boundary to far field operator and give a characterization  of the adjoint operator, following ideas of \cite{Hohage}, which is needed in the implementation of the regularized Newton method. In section 6, we present the inverse scattering algorithm in the special case of star-shaped obstacles.  Finally, in section 7  we show numerical experiments.

\section{The solution of the dielectric obstacle scattering problem} \label{ScatProb}
In this paper, we will assume that $\Gamma$ is the boundary of a smooth domain 
 $\Omega\subset\R^3$, which is diffeomorphic to a ball, so in particular $\Gamma$
is connected and simply connected.

\begin{notation}
{\rm We denote by $H^s(\Omega)$, $H^s_{\loc}(\overline{\Omegae})$ and $H^s(\Gamma)$ the standard (local in the case of the exterior domain) complex valued, Hilbertian Sobolev space of order $s\in\R$ defined on $\Omega$, $\overline{\Omegae}$ and $\Gamma$ respectively (with the convention $H^0=L^2$.)  Spaces of vector functions will be denoted by boldface letters, thus $\HH^s=(H^s)^3$. Moreover, $\HH^s_\tang(\Gamma):=\{\jj\in\HH^{s}(\Gamma);\;\jj\cdot\nn=0\}$ denotes the Sobolev space of tangential vector fields of order $s\in\R$. 
If $\Lambda$ is a differential operator, we write:
\begin{eqnarray*}
\HH^s(\Lambda,\Omega)& = &\{ \vv \in \HH^s(\Omega) : \Lambda \vv \in \HH^s(\Omega)\},\\
\HH^s_{\loc}(\Lambda,{\Omegae})&=& \{ \vv \in \HH^s_{\loc}(\overline{\Omegae}) : \Lambda \vv \in
\HH^s_{\loc}(\overline{\Omegae}) \}.\end{eqnarray*}
The space $\HH^s(\Lambda, \Omega)$ is endowed with the natural graph norm
$\|v\|_{\HH(\Lambda,\Omega)}:=(\|v\|_{L^2(\Omega)}^2 + \|\Lambda v\|_{L^2(\Omega)}^2)^{1/2}$. 
This defines in particular the Hilbert spaces $\HH^s(\Rot,\Omega)$ and $\HH^s(\Rot\Rot,\Omega)$ and the Fr\'echet spaces $\HH^s_{\loc}(\Rot,\Omegae)$ and $\HH^s_{\loc}(\Rot\Rot,\Omegae)$. When $s=0$ we omit the upper index $0$.\\
Analogously, we  introduce  for $s\in\R$ the Hilbert space 
\begin{eqnarray*}
  \HH_{\Div}^{s}(\Gamma)&=&\left\{ \jj\in
\HH^{s}_{\tang}(\Gamma);\;\Div_{\Gamma}\jj \in
H^{s}(\Gamma)\right\}
\end{eqnarray*}
endowed with the norm
$
||\cdot||_{ \HH_{\Div}^{s}(\Gamma)}=(||\cdot||_{\HH^{s}(\Gamma)}^2
+||\Div_{\Gamma}\cdot\;||_{H^{s}(\Gamma)}^2)^{1/2}.$}
\end{notation}

Recall that for a vector function $\uu\in\HH(\Rot,\Omega)\cap\HH(\Rot\Rot,\Omega)$ the traces $\nn\times\uu_{|\Gamma}$ and $\nn\times\Rot\uu_{|\Gamma}$ are in   $\HH_{\Div}^{-1/2}(\Gamma)$ (see e.g.\ \cite{Nedelec}).



\bigskip

It will be useful to simultaneously work with two different approaches  to solve the dielectric scattering 
problems described in \cite{MartinOla}. Both methods yield boundary integral equation systems of the second kind.

Let $\FS{\kappa}{\xb}=\dfrac{e^{i\kappa|\xb|}}{4\pi| \xb|}$
denote the fundamental solution of the Helmholtz equation $ {\Delta u + \kappa^2u =0}.
$
For any  solution $\EE^s$ to the Maxwell equation $\eqref{ME}$ in $\Omegae$ that satisfies the radiation condition \eqref{T3} the Stratton-Shu representation formula
\begin{equation*}\begin{split}\EE^{\rm s}(\xb)=&\;\frac{\mue}{\kappae^2}\int_{\Gamma}\Rot^{\xb}\Rot^{\xb}\left\{\FS{\kappae}{\xb-\yb}\left(\frac{1}{\mue}\nn(\yb)\times\Rot\EE^{\rm s}(\yb)\right)\right\}ds(\yb)\\&+\int_{\Gamma}\Rot^\xb\Big\{\FS{\kappae}{\xb-\yb}\big(\nn(\yb)\times\EE^{\rm s}(\yb)\big)\Big\}ds(\yb).\end{split}\end{equation*}
holds true. Analogously, for solutions $\EE^{\rm i}$ to Maxwell's equations 
$\eqref{ME}$ in the interior domain $\Omega$ the Stratton-Shu representation formula reads
\begin{equation}\label{Ei}\begin{split}\EE^{\rm i}(\xb)=&\;-\frac{\mui}{\kappai^2}\int_{\Gamma}\Rot^{\xb}\Rot^{\xb}\left\{\FS{\kappai}{\xb-\yb}\left(\frac{1}{\mui}\nn(\yb)\times\Rot\EE^{\rm i}(\yb)\right)\right\}ds(\yb)\\&-\int_{\Gamma}\Rot^\xb\Big\{\FS{\kappai}{\xb-\yb}\big(\nn(\yb)\times\EE^{\rm i}(\yb)\big)\Big\}ds(\yb)\,,
\qquad \xb\in\Omega\,.
\end{split}\end{equation}
By Green's second formula in $\Omega$ we have
\begin{equation}\label{Einc}
\begin{split}0=&\;\frac{\mue}{\kappae^2}\int_{\Gamma}\Rot^{\xb}\Rot^{\xb}\left\{\FS{\kappae}{\xb-\yb}\left(\frac{1}{\mue}\nn(\yb)\times\Rot\EE^{\rm inc}(\yb)\right)\right\}ds(\yb)\\
&+\int_{\Gamma}\Rot^\xb\Big\{\FS{\kappae}{\xb-\yb}\big(\nn(\yb)\times\EE^{\rm inc}(\yb)\big)\Big\}ds(\yb)\,,\qquad \xb\in\Omegae\,.
\end{split}\end{equation}
Adding \eqref{Ei} and \eqref{Einc} we obtain the following integral representation for the scattered wave
\begin{equation}\label{Es}\begin{split}\EE^{\rm s}(\xb)=&\;\frac{\mue}{\kappae^2}\int_{\Gamma}\Rot^{\xb}\Rot^{\xb}\left\{\FS{\kappae}{\xb-\yb}\left(\frac{1}{\mue}\nn(\yb)\times\Rot(\EE^{\rm s}+\EE^{\rm inc})(\yb)\right)\right\}ds(\yb)\\&+\int_{\Gamma}\Rot^\xb\Big\{\FS{\kappae}{\xb-\yb}\big(\nn(\yb)\times(\EE^{\rm s}+\EE^{\rm inc}(\yb)\big)\Big\}ds(\yb)
\,,\qquad \xb\in\Omegae\,.
\end{split}\end{equation}
By \eqref{Ei} and \eqref{Es}, one can see that the solution of the forward problem is uniquely determined by the knowledge of  the interior boundary values $\nn\times \EE^{\rm i}$ and $\tfrac{1}{\mui}\nn\times\Rot\EE^{\rm i}$ and the exterior boundary values $\nn\times(\EE^{\rm s}+\EE^{\rm inc})$ and $\tfrac{1}{\mue}\nn\times\Rot(\EE^{\rm s}+\EE^{\rm inc})$.
Thanks to the transmission conditions \eqref{T1} one can  reduce in several different ways the dielectric scattering problem to a system of two equations for the two unkowns  $\nn\times(\EE^{\rm s}+\EE^{\rm inc})$ and $\tfrac{1}{\mue}\nn\times\Rot(\EE^{\rm s}+\EE^{\rm inc})$.  The most attractive  boundary integral formulation of the problem via a direct method is M\"uller's one \cite{Muller} 
since it yields a uniquely solvable system of boundary integral equations of the second kind for all positive values of the dielectric constant,  

To derive the boundary integral formulation we introduce the single layer potential $C_{\kappa}$ and the double layer potential $M_{\kappa}$ in electromagnetic potential theory by 
\begin{eqnarray*}(M_{\kappa}\jj)(\xb)&=&-\displaystyle{\int_{\Gamma}\nn(\xb)\times\Rot^\xb\{2\FS{\kappa}{\xb-\yb}\jj(\yb)\}ds(\yb)},\\(C_{\kappa}\jj)(\xb)&=&-\frac{1}{\kappa}\displaystyle{\int_{\Gamma}\nn(\xb)\times\Rot^{\xb}\Rot^\xb\{2\FS{\kappa}{\xb-\yb}\jj(\yb)\}ds(\yb)}.
\end{eqnarray*}
The operator $M_{\kappa}:\HH^{-1/2}_{\Div}(\Gamma)\to\HH^{-1/2}_{\Div}(\Gamma)$ is compact and the operator $C_{\kappa}$ is of order +1  but bounded on $\HH^{-1/2}_{\Div}(\Gamma)$.
The  Calder\'on projectors for the time-harmonic Maxwell equation
\begin{equation}\label{MEgenerique}\Rot\Rot\EE-\kappa^2\EE=0\end{equation}
  are $P_{\kappa} = \Id +
A_{\kappa}$ and $P_{\kappa}^{\rm c} = \Id - A_{\kappa}$ where 
\begin{displaymath}
A_{\kappa} = \left(\begin{array}{ll} M_{\kappa} \ C_{\kappa} \\ C_{\kappa} \ M_{\kappa} \\ \end{array}\right ).
\end{displaymath}
We have $P_{\kappa}\circ P_{\kappa}^{\rm c} = 0$. This means that if $\EE_{|\Omega}\in\HH(\Rot,\Omega)$ solves  \eqref{MEgenerique} in $\Omega$, then 
\begin{equation}\label{P}P_{\kappa}\begin{pmatrix}\nn\times\EE\\\frac{1}{\kappa}\nn\times\Rot\EE\end{pmatrix}=2\begin{pmatrix}\nn\times\EE\\\frac{1}{\kappa}\nn\times\Rot\EE\end{pmatrix} \text{ and }\;P_{\kappa}^{\rm c}\begin{pmatrix}\nn\times\EE\\\frac{1}{\kappa}\nn\times\Rot\EE\end{pmatrix}=0,\end{equation}
and if $\EE_{|\Omegae}\in\HH_{\loc}(\Rot,\Omegae)$ solves  \eqref{MEgenerique} in $\Omegae$ and satisfies the Silver-M\"uller radiation condition, then 
\begin{equation}\label{Pc}P_{\kappa}\begin{pmatrix}\nn\times\EE\\\frac{1}{\kappa}\nn\times\Rot\EE\end{pmatrix}=0 \text{ and }\;P_{\kappa}^{\rm c}\begin{pmatrix}\nn\times\EE\\\frac{1}{\kappa}\nn\times\Rot\EE\end{pmatrix}=\begin{pmatrix}2\,\nn\times\EE\\\frac{2}{\kappa}\nn\times\Rot\EE\end{pmatrix}\,.\end{equation}
Now we set $$\uu^{\rm s}=\begin{pmatrix}\nn\times\EE^{\rm s}\\\frac{1}{\mue}\nn\times\Rot\EE^{\rm s}\end{pmatrix}, \quad\uu^{\rm inc}=\begin{pmatrix}\nn\times\EE^{\rm inc}\\\frac{1}{\mue}\nn\times\Rot\EE^{\rm inc}\end{pmatrix},\quad\uu^{\rm i}=\begin{pmatrix}\nn\times\EE^{{i}}\\\frac{1}{\mui}\nn\times\Rot\EE^{\rm i}\end{pmatrix},$$
and 
\begin{displaymath}
K_{\rm i} = \left(\begin{array}{ll} M_{\kappai} \ \frac{\mui}{\kappai}C_{\kappai} \\ \frac{\kappai}{\mui}C_{\kappai} \ M_{\kappai} \\ \end{array}\right )\qquad\text{and}\qquad K_{\rm e} = \left(\begin{array}{ll} M_{\kappae} \ \frac{\mue}{\kappae}C_{\kappae} \\ \frac{\kappae}{\mue}C_{\kappae} \ M_{\kappae} \\ \end{array}\right ).
\end{displaymath}
By virtue of \eqref{P} and \eqref{Pc} we  have 
\begin{align}\label{eq1}
0&=\begin{pmatrix}1&0\\0&\frac{\kappai}{\mui}\end{pmatrix}P_{\kappai}^{\rm c}\begin{pmatrix}\nn\times\EE^{\rm i}\\\frac{1}{\kappai}\nn\times\Rot\EE^{\rm i}\end{pmatrix}=\left(\Id-K_{\rm i}\right)\uu^{\rm i}\,,\\
\label{eq2}
2\uu^{\rm inc}&=
\begin{pmatrix}1&0\\0&\frac{\kappae}{\mue}\end{pmatrix}P_{\kappae}\begin{pmatrix}\nn\times(\EE^{\rm s}+\EE^{\rm inc})\\\frac{1}{\kappae}\nn\times\Rot(\EE^{\rm s}+\EE^{\rm inc})\end{pmatrix}=\left(\Id+K_{\rm e}\right)(\uu^{\rm s}+\uu^{\rm inc})\,,
\end{align}
and the transmission conditions give
\begin{equation}\label{eq3}\uu^{\rm i}=\uu^{\rm s}+\uu^{\rm inc}\,.\end{equation}
M\"uller's boundary integral formulation has to be solved for the unknown $\uu^{\rm s}+\uu^{\rm inc}$ and is obtained by 
plugging \eqref{eq3} into \eqref{eq1} and combining the equalities \eqref{eq1} and \eqref{eq2} as follows:
\begin{equation}\label{inteq1}\big(\Id+K_{\rm e}\big)(\uu^{\rm s}+\uu^{\rm inc})+\begin{pmatrix}\frac{\mue\kappai^2}{\mui\kappae^2}&0\\0&\frac{\mui}{\mue}\end{pmatrix}\big(\Id-K_{\rm i}\big)(\uu^{\rm s}+\uu^{\rm inc})=2\uu^{\rm inc}\end{equation}
This can be rewritten as
\begin{equation}\label{eq:directmethod}
\left\{\begin{pmatrix}1+\frac{\mue\kappai^2}{\mui\kappae^2}&0\\0&1+\frac{\mui}{\mue}\end{pmatrix}+\begin{pmatrix}M_{\kappae}-\frac{\mue\kappai^2}{\mui\kappae^2}M_{\kappai}&\frac{\mue}{\kappae^2}(\kappae C_{\kappae}-\kappai C_{\kappai})\\\frac{1}{\mue}(\kappae C_{\kappae}-\kappai C_{\kappai})&M_{\kappae}-\frac{\mui}{\mue}M_{\kappai}\end{pmatrix}\right\}(\uu^{\rm s}+\uu^{\rm inc})=2\uu^{\rm inc}.
\end{equation}
Since $\kappai C_{\kappai}-\kappae C_{\kappae}$ is compact on $\HH_{\Div}^{-1/2}(\Gamma)$, the integral operator associated to the integral equation \eqref{inteq1} is a Fredholm operator of index $0$  on the Hilbert space $\HH_{\Div}^{-1/2}(\Gamma)$.  The condition $\uu^{\rm inc}\in\big(\HH^{-1/2}_{\Div}(\Gamma)\big)^2$ guarantees that the solution to the integral equation is in $\big(\HH^{-1/2}_{\Div}(\Gamma)\big)^2$ too. 
\medskip
 
We present now an alternative approach via an indirect method in order to derive an other second kind system of integral equations (see \cite{MartinOla}).  It can be used to solve  electromagnetic transmission problem with general transmission conditions  of the form

\begin{equation}\label{T1generique}
\begin{aligned}
&\nn\times \EE^{\rm s}-\nn\times\EE^{\rm i}=\ff,\\
&\dfrac{1}{\mue}\nn\times\Rot\EE^{\rm s}-\dfrac{1}{\mui}\nn\times\Rot\EE^{\rm i}=\gg,
\end{aligned}
\end{equation} 
where $\ff,\gg\in\HH^{-1/2}_{\Div}(\Gamma)$ are given and is based on the layer ansatz 
\begin{align}\label{eq:ansatz}
\begin{aligned}\EE^{\rm s}(\xb)=&
\int_{\Gamma}\paren{\frac{\mue}{\kappae^2}
\Rot^{\xb}\Rot^{\xb}\Big\{\FS{\kappae}{\xb-\yb}\mm^{\rm s}(\yb)\Big\}
+\Rot^\xb\Big\{\FS{\kappae}{\xb-\yb}\jj^{\rm s}(\yb)\Big\}}ds(\yb)\\
\EE^{\rm i}(\xb)=&\int_{\Gamma}
\paren{\frac{\mui}{\kappai^2}\Rot^{\xb}\Rot^{\xb}\Big\{\FS{\kappai}{\xb-\yb}\mm^{\rm i}(\yb)\Big\}
+\Rot^\xb\Big\{\FS{\kappai}{\xb-\yb}\jj^{\rm i}(\yb)\Big\}}ds(\yb)
\end{aligned}\end{align}
 where $\jj^{\rm s},\mm^{\rm s},\jj^{\rm i},\mm^{\rm i}$ are tangential densities in $\HH^{-1/2}_{\Div}(\Gamma)$. By virtue of \eqref{P} and \eqref{Pc} and the jump relations of the electromagnetic potentials 
we have $$P_{\kappai}\begin{pmatrix}\jj^{\rm i}\\\frac{\mui}{\kappai}\mm^{\rm i}\end{pmatrix}=\begin{pmatrix}-2\,\nn\times\EE^{\rm i}\\-\frac{2}{\kappai}\nn\times\Rot\EE^{\rm i}\end{pmatrix} \text{ and } \quad P^{\rm c}_{\kappae}\begin{pmatrix}\jj^{\rm s}\\\frac{\mue}{\kappae}\mm^{\rm s}\end{pmatrix}=\begin{pmatrix}2\,\nn\times\EE^{\rm s}\\\frac{2}{\kappae}\nn\times\Rot\EE^{\rm s}\end{pmatrix}\,.$$
The transmission conditions yields
 $$\begin{pmatrix}1&0\\0&\frac{\kappai}{\mui}\end{pmatrix}P_{\kappai}\begin{pmatrix}\jj^{\rm i}\\\frac{\mui}{\kappai}\mm^{\rm i}\end{pmatrix}+\begin{pmatrix}1&0\\0&\frac{\kappae}{\mue}\end{pmatrix}P^{\rm c}_{\kappae}\begin{pmatrix}\jj^{\rm s}\\\frac{\mue}{\kappae}\mm^{\rm s}\end{pmatrix}=2 \begin{pmatrix}\ff\\\gg\end{pmatrix}\,.$$
This equation is equivalent to 
 $$\left(\Id+K_{\rm i}\right)\begin{pmatrix}\jj^{\rm i}\\\mm^{\rm i}\end{pmatrix}+\left(\Id-K_{\rm e}\right)
\begin{pmatrix}\jj^{\rm s}\\ \mm^{\rm s}\end{pmatrix}
=2\begin{pmatrix}\ff\\\gg\end{pmatrix}\,.$$
We set $\jj^{\rm i}=\frac{\mui}{\mue}\jj$, $\jj^{\rm s}=\jj$ and $\mm^{\rm i}=\frac{\mue\kappai^{2}}{\mui\kappae^2}\mm$, $\mm^{\rm s}=\mm$.
Then we arrive at the following system  of integral equations first obtained by Ola and Martin \cite{MartinOla}:
\begin{equation}\label{inteq2}\left(\Id+K_{\rm i}\right)\begin{pmatrix}\frac{\mui}{\mue}&0\\0&\frac{\mue\kappai^2}{\mui\kappae^2}\end{pmatrix}\begin{pmatrix}\jj\\\mm\end{pmatrix}+\left(\Id-K_{\rm e}\right)\begin{pmatrix}\jj\\\mm\end{pmatrix}=\begin{pmatrix}2\ff\\2\gg\end{pmatrix}
\end{equation}
Interchanging the order of the entries of the vectors $\smat{\jj\\ \mm}$ and
$\smat{\ff\\ \gg}$ it can be rewritten as
\begin{align}\label{eq:indirectmethod}
&\left\{\begin{pmatrix}
1+\frac{\mue\kappai^2}{\mui\kappae^2}&0\\0&1+\frac{\mui}{\mue}
\end{pmatrix} -\begin{pmatrix}
M_{\kappae}-\frac{\mue\kappai^2}{\mui\kappae^2}M_{\kappai} & \frac{1}{\mue}(\kappae C_{\kappae}-\kappai C_{\kappai}) \\
\frac{\mue}{\kappae^2}(\kappae C_{\kappae}-\kappai C_{\kappai}) & M_{\kappae}-\frac{\mui}{\mue}M_{\kappai}
\end{pmatrix}\right\}
\begin{pmatrix}\mm\\ \jj\end{pmatrix}
=2\begin{pmatrix}\gg\\ \ff\end{pmatrix}.
\end{align}\medskip

\begin{remark}\label{rem:transposed}
{\rm Let us compare the system matrix $K_{\rm DM}$ of the direct method in \eqref{eq:directmethod}
and the system matrix $K_{\rm IM}$ of the indirect method in \eqref{eq:indirectmethod}. 
Let $K^\top:=\overline{K^*\overline{f}}$ denote the adjoint of an operator $K$ with respect 
to the bilinear rather than the sesquilinear $L^2$ product and recall that 
$M_{\kappa}^{\top} =  RM_{\kappa}R$ and $C_{\kappa}^{\top} =  RC_{\kappa}R$ with
$R\ff = \nn\times\ff$. 
Using this and the identity $R^2=-I$ we find that
\begin{equation}\label{eq:transposed}
K_{\rm IM}^{\top} = - \smat{R&0\\ 0&R}
K_{\rm DM}  \smat{R&0\\ 0&R}\,.
\end{equation}
This relation is useful in the context of iterative regularization methods for the
inverse problem where both systems with the operator $K_{\rm DM}$ and with the operator
$K_{\rm IM}$ have to be solved in each iteration step. If these operators are essentially represented by transposed matrices, only one matrix has to be set up and only one LU
decomposition has to be computed if the discrete linear systems are solved by 
Gaussian elimination.} 
\end{remark}

It follows from the representation formula \eqref{Es} and the ansatz \eqref{eq:ansatz} that 
the far field pattern can be computed via the integral representation formulas
$$
\begin{array}{lcll}\EE^{\infty}&=&\displaystyle{\FFop(\uu^{\rm s}+\uu^{\rm inc})}
&\text{ if one solves the equation \eqref{inteq1}}\text{ or}\\
\EE^{\infty}&=&\displaystyle{\FFop\smat{\jj\\ \mm}}
&\text{ if one solves the equation \eqref{eq:indirectmethod}},\end{array}
$$
using the far field operator $\displaystyle{\FFop}:\big(\HH^{-1/2}_{\Div}(\Gamma)\big)^2\to\LL^2_{\tang}(\S^2)$ defined for $\xhat\in\S^2$ by  
\begin{equation*}\begin{split}\FFop\begin{pmatrix}\jj\\ \mm\end{pmatrix}(\xhat)
=\frac{\mue}{4\pi}\int_{\Gamma}e^{-i\kappae\xhat\cdot \yb}
\Big(\xhat\times\mm(\yb)\times\xhat\Big)ds(\yb)
+\frac{i\kappae}{4\pi}\int_{\Gamma}e^{-i\kappae\xhat\cdot \yb}\big(\xhat\times\jj(\yb)\big)ds(\yb)\,.
\end{split}\end{equation*}

\begin{notation}\label{not:triple_cross}
{\rm Although in general ${\bf a}\times ({\bf b}\times {\bf c})=
({\bf a}\cdot{\bf c}){\bf b}-({\bf a}\cdot{\bf b}){\bf c}$ 
is different from $({\bf a}\times {\bf b})\times {\bf c}$
for ${\bf a},{\bf b},{\bf c}\in\mathbb{R}^3$,
both expressions coincide for ${\bf a}={\bf c}$. The vector  
${\bf a}\times ({\bf b}\times {\bf a})= ({\bf a}\times {\bf b})\times {\bf a}$
is the orthogonal projection of $\bf b$ onto the plane orthogonal to $\bf a$
and is denoted by ${\bf a}\times {\bf b}\times {\bf a}$.}
\end{notation}

 \section{ A high-order spectrally accurate  algorithm}
The first step in the derivation of our algorithm is a transformation of the integral equations on
$\Gamma$ derived above to integral equations on the unit sphere $\S^2$ of $\R^3$. 
We denote by $\theta,\phi$ the spherical coordinates of any point $\xhat\in\S^2$, i.e.
\begin{equation}\label{eq:spherical_coo}
\xhat=\psi(\theta,\phi)=\begin{pmatrix}\sin\theta\cos\phi\\\sin\theta\sin\phi\\\cos\theta\end{pmatrix},\quad(\theta,\phi)\in\,]0;\pi[\times[0;2\pi[\,\cup\{(0,0);(0,\pi)\}\,.
\end{equation}
The tangent and the cotangent planes at any point $\xhat=\psi(\theta,\phi)\in\S^2$ is generated by the unit vectors 
\(
\ee_{\theta}=\frac{\partial\psi}{\partial\theta}(\theta,\phi)
\) and \(
\ee_{\phi}=\frac{1}{\sin\theta}\frac{\partial\psi}{\partial\phi}(\theta,\phi).
\)
The triplet $(\xhat,\ee_{\theta},\ee_{\phi})$ forms an 
orthonormal system. 
The determinant of the Jacobian is $J_{\psi}(\theta,\phi)=\sin\theta$. 

Let $\qq:\S^2\rightarrow\Gamma$ be a parametrization of class $\mathscr{C}^1$ at least. 
We will use the notation of the appendix. 
The total derivative $[\D \qq(\xhat)]$  maps the tangent plane $\TT_{\xhat}$ to $\S^2$ at the point $\xhat$ onto the tangent plane $\TT_{\qq(\xhat)}$ to $\Gamma$ at the point $\qq(\xhat)$. 
The latter  is generated by the vectors
\begin{align*}
\tt_{1}(\xhat)&=\ee_{1}(\qq(\xhat))=\frac{\partial \qq\circ\psi}{\partial\theta}\circ\psi^{-1}
=[\D \qq(\xhat)]\ee_{\theta}\,,\\
\tt_{2}(\xhat)&=\frac{1}{J_{\psi}\circ\psi^{-1}(\xhat)}\ee_{2}(\qq(\xhat))
=\frac{1}{J_{\psi}\circ\psi^{-1}(\xhat)}\frac{\partial \qq\circ\psi}{\partial\phi}\circ\psi^{-1}
=[\D \qq(\xhat)]\ee_{\phi}\,.
\end{align*}
The determinant $J_{\qq}$ of the Jacobian of the change of variables $\qq:\S^2\to\Gamma$
and the normal vector $\nn\circ\qq$ can be computed via the formulas 
$$J_{\qq}=\frac{J_{\qq\circ\psi}\circ\psi^{-1}}{ J_{\psi}\circ\psi^{-1}}=\big|\tt_{1}\times\tt_{2}\big| \text{ and } \nn\circ\qq=\frac{(\ee_{1}\circ\qq)\times(\ee_{2}\circ\qq)}{J_{\qq\circ\psi}\circ\psi^{-1}}=\frac{\tt_{1}\times\tt_{2}}{J_{\qq}}.$$
The parametrization $\qq:\S^2\to\Gamma$ being a diffeomorphism, we set $[{\D \qq(\xhat)}]^{-1}=[{\D \qq^{-1}]\circ\qq(\xhat)}$. The transposed matrix $[\adjoint{\D \qq(\xhat)}]^{-1}$  maps the cotangent plane $\TT^{*}_{\xhat}$ to $\S^2$ at the point $\xhat$ onto the cotangent plane $\TT^{*}_{\qq(\xhat)}$ to $\Gamma$ at the point $\qq(\xhat)$. The latter is generated by the vectors
\begin{align*}
\tt^{1}(\xhat)&=\ee^{1}(\qq(\xhat))=\frac{1}{J_{\qq\circ\psi}\circ\psi^{-1}}\ee_{2}(\qq(\xhat))\times\nn(\qq(\xhat))=\frac{\tt_{2}(\qq(\xhat))\times\nn(\qq(\xhat))}{J_{\qq}(\xhat)}=[\adjoint{\D \qq(\xhat)}]^{-1}\ee_{\theta}\,,\\
\tt^{2}(\xhat)&
=J_{\psi}\circ\psi^{-1}(\xhat)\,\ee^{2}(\qq(\xhat))=\frac{\nn(\qq(\xhat))\times\tt_{1}(\qq(\xhat))}{J_{\qq}(\xhat)}
=[\adjoint{\D \qq(\xhat)}]^{-1}\ee_{\phi}\,.
\end{align*}
In view of the formulas \eqref{G}-\eqref{R}, it is straightforward to deduce the following transformation formulas for the surface differential operators: 
\begin{align}\label{transformSurfaceDOP}
\begin{aligned}
&(\grad_{\Gamma}u)\circ\qq=[\adjoint{\D \qq}]^{-1}\grad_{\S^2}(u\circ\qq),
&&(\Rot_{\Gamma}u)\circ\qq=\frac{1}{J_{\qq}}[\D\qq]\Rot_{\S^2}(u\circ\qq),\\
&(\Div_{\Gamma}\vv)\circ\qq=\frac{1}{J_{\qq}}\Div_{\S^2}\big(J_{\qq}\,[\D \qq]^{-1}(\vv\circ\qq)\big),
&&(\rot_{\Gamma}\ww)\circ\qq=\frac{1}{J_{\qq}}\rot_{\S^2}\big([\adjoint{\D \qq}](\ww\circ\qq)\big).
\end{aligned}
\end{align}
From this we can now introduce a boundedly invertible operator from $\HH^{-1/2}_{\Div}(\Gamma)$ to $\HH^{-1/2}_{\Div}(\S^2)$. We first recall that $\HH^{-1/2}_{\Div}(\Gamma)$ admits the Hodge decomposition \cite{dlB}
\begin{equation}\label{eq:Hodge_decomp}
\HH^{-1/2}_{\Div}(\Gamma)=\grad_{\Gamma}H^{\frac{3}{2}}(\Gamma)\oplus\Rot_{\Gamma}H^{\frac{1}{2}}(\Gamma)
\end{equation}
provided that the surface $\Gamma$ is smooth and simply connected, which we have
assumed. A first transformation, which intertwines with the Hodge deomposition, is the following:
$$
\begin{array}{cccl}
\HH^{-\frac{1}{2}}_{\Div}(\Gamma)&\longrightarrow& &\quad\HH^{-\frac{1}{2}}_{\Div}(\S^2)
\vspace{1mm}\\
\jj=\grad_{\Gamma}\,p_{1}+\Rot_{\Gamma}p_{2}&\mapsto&\js&=[\adjoint{\D \qq}](\grad_{\Gamma}\,p_{1})\circ\qq
+J_{\qq}[\D\qq]^{-1}(\Rot_{\Gamma}p_{2})\circ\qq
\\
&&&=\grad_{\S^2}\,(p_{1}\circ\qq)+\Rot_{\S^2}(p_{2}\circ\qq).
\end{array}
$$
This transformation was first considered by Costabel and Le Lou\"er \cite{CostabelLeLouer2} in the context of the shape differentiability analysis of the boundary integral operators $M_{\kappa}$ and $C_{\kappa}$. However, for the numerical solution of boundary integral equations it is
inconvenient as it requires explicit knowledge of the Hodge decomposition. Therefore, 
we will use a second tranformation, the so-called Piola transform of $\qq$, introduced in the
following lemma:

\begin{lemma}
The linear mapping
\begin{equation}
\label{Trdiv}
\begin{array}{ccc}
\Pq:\HH^{-\frac{1}{2}}_{\Div}(\Gamma)&\longrightarrow& \HH^{-\frac{1}{2}}_{\Div}(\S^2)
\\
\jj&\mapsto&\js=J_{\qq}[\D\qq]^{-1}(\jj\circ\qq).
\end{array}
\end{equation}
is well-defined and bounded and has a bounded inverse 
$\Pq^{-1}:\HH^{-1/2}_{\Div}(\S^2)\to\HH^{-1/2}_{\Div}(\Gamma)$, 
$\Pq^{-1}\js = (\frac{1}{J_q}[\D\qq]\js)\circ \qq^{-1}$. 
\end{lemma}

\begin{proof}
To see that $\js$ belongs to $\HH^{-1/2}_{\Div}(\S^2)$, write
$\jj=\grad_{\Gamma}\,p_{1}+\Rot_{\Gamma}p_{2}$ with $p_1\in H^{3/2}(\Gamma)$
and $p_2\in H^{1/2}(\Gamma)$ according to \eqref{eq:Hodge_decomp} and note that
$$
\Div_{\S^2} J_{\qq}[\D\qq]^{-1}(\Rot_{\Gamma}p_2)\circ\qq 
= \Div_{\S^2}\Rot_{\S^2}(p_2\circ\qq)=0$$ 
using \eqref{eq:chain}.
As $ J_{\qq}[\D\qq]^{-1} (\grad_{\Gamma}p_1)\circ\qq \in \HH^{1/2}(\S^2)$ it follows that 
$\Div_{\S^2} \js  \in H^{-1/2}(\Gamma)$, and $\jj\in H^{-1/2}(\Gamma)$ implies
$\js\in H^{-1/2}(\S^2)$. The boundedness of $\Pq$ is obvious.
The proof for $\Pq^{-1} = \mathcal{P}_{\qq^{-1}}$ is analogous. 
\end{proof}

We construct our spectral method by replacing the boundary integral operators 
$M_{\kappa}$ and $C_{\kappa}$ in \eqref{inteq1} and \eqref{inteq2} by the operators 
\[
\mathcal{M}_{\kappa}:= \Pq M_{\kappa}\Pq^{-1}\qquad\mbox{and}\qquad 
\mathcal{C}_{\kappa}:= \Pq C_{\kappa}\Pq^{-1}
\]
which map $\HH^{-1/2}_{\Div}(\S^2)$ boundedly into itself and are given by
\begin{align*}
\mathcal{M}_{\kappa}\,\js
&=-J_{q}[\D \qq]^{-1}\displaystyle{\int_{{\S^2}}(\nn\circ\qq)\times\Rot\{2\FS{\kappa}{\qq(\cdot)-\qq(\yhat)}[\D \qq(\yhat)]\js(\yhat)\}ds(\yhat)}\,,\\
\mathcal{C}_{\kappa}\,\js
&=-\kappa\, J_{q}[\D \qq]^{-1}\int_{{\S^2}}(\nn\circ\qq)\times\{2\FS{\kappa}{\qq(\cdot)-\qq(\yhat)}[\D \qq(\yhat)]\js(\yhat)\}ds(\yhat)\\
&-\dfrac{1}{\kappa}J_{q}[\D \qq]^{-1}\int_{{\S^2}}(\nn\circ\qq)\times\grad\Div\{2\FS{\kappa}{\qq(\cdot)-\qq(\yhat)}[\D\qq(\yhat)]\js(\yhat)\}ds(\yhat)\,.
\end{align*}
The new unknowns  will be two tangential vector densities in $\HH^{-1/2}_{\Div}(\S^2)$ obtained by applying the operator \eqref{Trdiv} to the unknowns  in \eqref{inteq1} and \eqref{inteq2}. 
  
 In our case we have to implement the  compact operators $\mathcal{M}_{\kappai}$, $\mathcal{M}_{\kappae}$ and 
$\kappae\mathcal{C}_{\kappae}-\kappai\mathcal{C}_{\kappai}$. To this end, we first split their kernels into a smooth and a weakly singular part. 
 We introduce the functions  
\begin{align*} 
\mathcal{S}_{1}(\qq;\kappa,\xhat,\yhat)&=\frac{1}{2\pi}\cos(\kappa|\qq(\xhat)-\qq(\yhat)|),\\
 \mathcal{S}_{2}(\qq;\kappa,\xhat,\yhat)&=\frac{1}{2\pi}\left\{\begin{array}{ll}\dfrac{\sin(\kappa|\qq(\xhat)-\qq(\yhat)|)}{|\qq(\xhat)-\qq(\yhat)|}&\xhat\not=\yhat,\\\kappa&\xhat=\yhat\end{array}\right.
\quad\mbox{and} \\ 
R(\qq;\xhat,\yhat)&=\frac{|\xhat-\yhat|}{|\qq(\xhat)-\qq(\yhat)|}\,.
\end{align*}
Then $\mathcal{M}_{\kappa}$ can be rewritten as
$$\mathcal{M}_{\kappa}\,\js(\xhat)=\int_{\S^2}\frac{R(\qq;\xhat,\yhat)}{|\xhat-\yhat|}\mathcal{M}_{1}(\qq;\kappa,\xhat,\yhat)\js(\yhat)ds(\yhat)+i\int_{\S^2}\mathcal{M}_{2}(\qq;\kappa,\xhat,\yhat)\js(\yhat)ds(\yhat)$$
where $\mathcal{M}_{1}(\qq;\kappa,\xhat,\yhat)$ and $\mathcal{M}_{2}(\qq;\kappa,\xhat,\yhat)$ are $3\times3$ matrices given by
\begin{align*}
\mathcal{M}_{1}(\qq;\kappa,\xhat,\yhat)&=\left(\frac{\mathcal{S}_{1}(\qq;\kappa,\xhat,\yhat)}{|\qq(\xhat)-\qq(\yhat)|^2}+\kappa \mathcal{S}_{2}(\qq;\kappa,\xhat,\yhat)\right)\mathcal{V}(\qq,\xhat,\yhat)\,, \\
\mathcal{M}_{2}(\qq;\kappa,\xhat,\yhat)&=\dfrac{\mathcal{S}_{2}(\qq;\kappa,\xhat,\yhat)-\kappa\mathcal{S}_{1}(\qq;\kappa,\xhat,\yhat)}{|\qq(\xhat)-\qq(\yhat)|^2}\mathcal{V}(\qq;\xhat,\yhat)
\end{align*}
with 
\begin{align*}
\mathcal{V}(\qq;\xhat,\yhat)&
=\tt_{2}(\xhat)\cdot\big(\tt_{1}(\yhat)\times\big(\qq(\xhat)-\qq(\yhat)\big)\big)
\,\ee_{\theta}(\xhat)\otimes\ee_{\theta}(\yhat)\\
&+\tt_{2}(\xhat)\cdot\big(\tt_{2}(\yhat)\times\big(\qq(\xhat)-\qq(\yhat)\big)\big)
\,\ee_{\theta}(\xhat)\otimes\ee_{\phi}(\yhat)\\
&-\tt_{1}(\xhat)\cdot\big(\tt_{1}(\yhat)\times\big(\qq(\xhat)-\qq(\yhat)\big)\big)
\,\ee_{\phi}(\xhat)\otimes\ee_{\theta}(\yhat)\\
&-\tt_{1}(\xhat)\cdot\big(\tt_{2}(\yhat)\times\big(\qq(\xhat)-\qq(\yhat)\big)\big)
\,\ee_{\phi}(\xhat)\otimes\ee_{\phi}(\yhat),
\end{align*}
The operator $\kappae\mathcal{C}_{\kappae}-\kappai\mathcal{C}_{\kappai}$ can be rewritten as 
$$\begin{array}{lcl}
\paren{\kappae\mathcal{C}_{\kappae}-\kappai\mathcal{C}_{\kappai}}\,\js(\xhat)
&=&\displaystyle{\int_{\S^2}\dfrac{R(\qq;\xhat,\yhat)}{|\xhat-\yhat|}
\left(\mathcal{C}_{1}(\qq;\kappae,\xhat,\yhat)-\mathcal{C}_{1}(\qq;\kappai,\xhat,\yhat)\right)\js(\yhat)ds(\yhat)}\vspace{2mm}\\&&+\;i\displaystyle{\int_{\S^2}\left(\mathcal{C}_{2}(\qq;\kappae,\xhat,\yhat)-\mathcal{C}_{2}(\qq;\kappai,\xhat,\yhat)\right)\js(\yhat)ds(\yhat)}\end{array}$$
where $\mathcal{C}_{1}(\qq;\kappa,\xhat,\yhat)$ and $\mathcal{C}_{2}(\qq;\kappa,\xhat,\yhat)$ are $3\times3$ matrices given by
$$\begin{array}{lcl}\mathcal{C}_{1}(\qq;\xhat,\yhat)&=&\left(\kappa^2\mathcal{S}_{1}(\qq;\kappa,\xhat,\yhat)-\dfrac{\mathcal{S}_{1}(\qq;\kappa,\xhat,\yhat)}{|\qq(\xhat)-\qq(\yhat)|^2}-\kappa \mathcal{S}_{2}(\qq;\kappa,\xhat,\yhat)\right)\mathcal{V}_{1}(\qq;\xhat,\yhat)
\vspace{1mm}\\
&&\hspace{-1cm}+\left(\dfrac{\mathcal{S}_{1}(\qq;\kappa,\xhat,\yhat)}{|\qq(\xhat)-\qq(\yhat)|^2}\left(-\kappa^2+\dfrac{3}{|\qq(\xhat)-\qq(\yhat)|^2}\right)+3\kappa \dfrac{\mathcal{S}_{2}(\qq;\kappa,\xhat,\yhat)}{|\qq(\xhat)-\qq(\yhat)|^2}\right)\mathcal{V}_{2}(\qq;\xhat,\yhat)\,,
\end{array}$$
$$\begin{array}{lcl}\mathcal{C}_{2}(\qq;\xhat,\yhat)&=&\left(\kappa^2\mathcal{S}_{2}(\qq;\kappa,\xhat,\yhat)-\dfrac{\mathcal{S}_{2}(\qq;\kappa,\xhat,\yhat)-\kappa\mathcal{S}_{1}(\qq;\kappa,\xhat,\yhat)}{|\qq(\xhat)-\qq(\yhat)|^2}\right)\mathcal{V}_{1}(\qq;\xhat,\yhat)
\vspace{1mm}\\
&&\hspace{-1cm}-\left(3\kappa\dfrac{\mathcal{S}_{1}(\qq;\kappa,\xhat,\yhat)}{|\qq(\xhat)-\qq(\yhat)|^4}-\dfrac{\mathcal{S}_{2}(\qq;\kappa,\xhat,\yhat)}{|\qq(\xhat)-\qq(\yhat)|^2}\left(-\kappa^2+\dfrac{3}{|\qq(\xhat)-\qq(\yhat)|^2}\right)\right)\mathcal{V}_{2}(\qq;\xhat,\yhat)\,,
\end{array}$$
with 
\begin{align*}
\mathcal{V}_{1}(\qq;\kappa,\xhat,\yhat)&=
(\tt_{2}(\xhat)\cdot\tt_{1}(\yhat))
\,\ee_{\theta}(\xhat)\otimes\ee_{\theta}(\yhat)
+(\tt_{2}(\xhat)\cdot\big(\tt_{2}(\yhat)\big)) 
\,\ee_{\theta}(\xhat)\otimes\ee_{\phi}(\yhat)\\
&-(\tt_{1}(\xhat)\cdot\tt_{1}(\yhat))
\,\ee_{\phi}(\xhat)\otimes\ee_{\theta}(\yhat)
-(\tt_{1}(\xhat)\cdot\tt_{2}(\yhat))
\,\ee_{\phi}(\xhat)\otimes\ee_{\phi}(\yhat),\\
\mathcal{V}_{2}(\qq;\kappa,\xhat,\yhat)&=
\Big(\tt_{2}(\xhat)\cdot\big(\qq(\xhat)-\qq(\yhat)\big)\Big)\Big(\tt_{1}(\yhat)\cdot\big(\qq(\xhat)-\qq(\yhat)\big)\Big)
\,\ee_{\theta}(\xhat)\otimes\ee_{\theta}(\yhat)\\
&+\Big(\tt_{2}(\xhat)\cdot\big(\qq(\xhat)-\qq(\yhat)\big)\Big)\Big(\tt_{2}(\yhat)\cdot\big(\qq(\xhat)-\qq(\yhat)\big)\Big)
\,\ee_{\theta}(\xhat)\otimes\ee_{\phi}(\yhat)\\
&-\Big(\tt_{1}(\xhat)\cdot\big(\qq(\xhat)-\qq(\yhat)\big)\Big)\Big(\tt_{1}(\yhat)\cdot\big(\qq(\xhat)-\qq(\yhat)\big)\Big)
\,\ee_{\phi}(\xhat)\otimes\ee_{\theta}(\yhat)\\
&-\Big(\tt_{1}(\xhat)\cdot\big(\qq(\xhat)-\qq(\yhat)\big)\Big)\Big(\tt_{2}(\yhat)\cdot\big(\qq(\xhat)-\qq(\yhat)\big)\Big)
\,\ee_{\phi}(\xhat)\otimes\ee_{\phi}(\yhat)\,.
\end{align*}
\medskip

Next, we introduce a change of coordinate system  in order to move all the singularities in the weakly singular integrals to  only one point that is chosen
to be the North pole. For $\xhat\in\S^2$ we consider an orthogonal transformation $T_{\xhat}$ which maps $\xhat$ onto the North pole denoted by  ${\nhat}$. If $\xhat=\xhat(\theta,\phi)$ then 
$T_{\xhat}:=P(\phi)Q(-\theta)P(-\phi)$ where $P$ and $Q$ are defined by 
$$P(\phi)=\left(\begin{matrix}\cos\phi&-\sin\phi&0\\\sin\phi&\cos\phi&0\\0&0&1\end{matrix}\right)
\quad\text{ and }\quad
Q(\theta)=\left(\begin{matrix}\cos\theta&0&\sin\phi\\0&1&0\\-\sin\phi&0&\cos\phi\end{matrix}\right).$$
We also introduce an induced linear tranformation $\mathcal{T}_{\xhat}$ defined by $\mathcal{T}_{\xhat}u(\yhat)=u(T_{\xhat}^{-1}\yhat)$  and we still denote by $\mathcal{T}_{\xhat}$ its bivariate analogue $\mathcal{T}_{\xhat}v(\yhat_{1},\yhat_{2})=v(T_{\xhat}^{-1}\yhat_{1},T_{\xhat}^{-1}\yhat_{2})$. If we write $\zhat=T_{\xhat}\,\yhat$ then we have the identity
$$|\xhat-\yhat|=|T_{\xhat}^{-1}(\nhat-\zhat)|=|\nhat-\zhat|.$$
The boundary integral operator $\mathcal{M}_{\kappa}$ can be rewritten in the form:
\begin{equation}\label{splitkernel}
\mathcal{M}_{\kappa}\,\js(\xhat)=\int_{\S^2}
\!\!\left(
\frac{\mathcal{T}_{\xhat}R(\qq;\nhat,\zhat)}{|\nhat-\zhat|}\mathcal{T}_{\xhat}\mathcal{M}_{1}(\qq;\kappa,\nhat,\zhat)\mathcal{T}_{\xhat}\js(\zhat)
+ i 
\mathcal{T}_{\xhat}\mathcal{M}_{2}(\qq;\kappa,\nhat,\zhat)\mathcal{T}_{\xhat}\js(\zhat)\!\right)ds(\zhat),
\end{equation}
and it can be shown that $(\theta',\phi')\mapsto \mathcal{T}_{\xhat}R(\qq;\nhat,\zhat(\theta',\phi'))\mathcal{T}_{\xhat}\mathcal{M}_{1}(\qq;\kappa,\nhat,\zhat(\theta',\phi'))$ is smooth.  An important point is that the  singularity
$\frac{1}{|\nhat-\zhat(\theta',\phi')|}=\frac{1}{2\sin\tfrac{\theta'}{2}}$
is cancelled out by the surface element $ds(\zhat)=\sin\theta' d\theta' d\phi'$. We proceed in the same way for the operator $(\kappae\mathcal{C}_{\kappae}-\kappai\mathcal{C}_{\kappai})$.\medskip

To solve the parametrized boundary integral equation systems we extend 
the spectral algorithm of Ganesh and Graham \cite{GaneshGraham}  to the vector case,  which ensures spectrally accurate convergence of the discrete solution for second kind scalar integral equations. With an alternative method this was done by
Ganesh and Hawkins \cite{GaneshHawkins3} for the perfect conductor problem.
For both of the boundary integral equation systems, it consists in the approximation of the (two) equations   in the subspace $\H_{n}\subset\HH^{-1/2}_{\Div}(\S^2)$ of finite dimension  $2(n+1)^2-2$ generated by the orthonormal basis of tangential vector spherical harmonics (see Appendix \ref{app:spherical_harmonics})   of degree $\leq n\in\N$.

 The numerical scheme  is based on a quadrature formula over the unit sphere of the
form
\begin{equation}\label{rule}\int_{\S^2}u(\xhat)ds(\xhat)
\approx\sum_{\rho=0}^{2n+1}\sum_{\tau=1}^{n+1}\mu_{\rho}\nu_{\tau}u(\xhat(\theta_{\tau},\phi_{\rho}))\,.
\end{equation}
Here $\theta_{\tau}=\arccos \zeta_{\tau}$ where $\zeta_{\tau}$, for $\tau=1,\hdots,n+1$, are the zeros of the Legendre polynomial $P^0_{n+1}$ of degree $n+1$ and $\nu_{\tau}$, for $\tau=1,\hdots,n+1$, are the corresponding Gauss-Legendre weights 
and $$\mu_{\rho}=\frac{\pi}{n+1},\quad\phi_{\rho}=\frac{\rho\pi}{n+1},\; \text{ for }\rho=0,\hdots,2n+1.$$
 The formula \eqref{rule} is exact for the spherical polynomials of order 
$\leq 2n+1$ (see \cite{Pieper}). Here and in the following we use the notation 
$\xhat_{\rho\tau}=\xhat(\theta_{\tau},\phi_{\rho})$.

 The main ingredient of the method is the fact that the scalar spherical harmonics (see  Appendix \ref{app:spherical_harmonics}) are eigenfunctions of the single layer potential on the sphere \cite{ColtonKress}: 
  \begin{equation}\label{ingredient}\int_{\S^2}\frac{1}{|\xhat-\yhat|}Y_{l,j}(\yhat)ds(\yhat)=\frac{4\pi}{2l+1}Y_{l,j}(\xhat), \quad\text{ for all }\xhat\in\S^2,\end{equation}
  and that we have for $l\geq1$
  $$\int_{\S^2}Y_{l,j}(\yhat)ds(\yhat)=0.$$
Let $\P_{n}$ denotes the space of all scalar spherical polynomials of degree  
$\leq n$.
We introduce a projection operator $\mathscr{L}_{n}$ onto $\P_{n}^3$ defined by 
$$\mathscr{L}_{n}\uu=\sum_{l=0}^n\sum_{j=-l}^l\begin{pmatrix}(u_{1},Y_{lj})_{n}\\(u_{2},Y_{lj})_{n}\\(u_{3},Y_{lj})_{n}\end{pmatrix}Y_{lj}
\quad \text{where}\quad 
(\varphi_{1},\varphi_{2})_{n}=\sum_{\rho=0}^{2n+1}\sum_{\tau=1}^{n+1}\mu_{\rho}\nu_{\tau}\varphi_{1}(\xhat_{\rho\tau})\overline{\varphi_{2}(\xhat_{\rho\tau})}
$$
and $\uu=\transposee{(u_{1},u_{2},u_{3})}$. Moreover, we introduce 
a projection operator $\L_{n}$ onto $\H_{n}$ by
$$\L_{n}\vv=\sum_{i=1}^2\sum_{l=1}^n\sum_{j=-l}^l(\vv|\Y^{(i)}_{lj})_{n}\,\Y^{(i)}_{lj}
\quad\text{where}\quad
(\vv_{1}|\vv_{2})_{n}=\sum_{\rho=0}^{2n+1}\sum_{\tau=1}^{n+1}\mu_{\rho}\nu_{\tau}\vv_{1}(\xhat_{\rho\tau})\cdot\overline{\vv_{2}(\xhat_{\rho\tau})}\,.
$$
In a first step, the operator $\mathcal{M}_{\kappa}^{\qq}$ is approximated by 
\begin{equation*}\begin{split}\mathcal{M}_{\kappa,n'}\,\js(\xhat)=&\int_{\S^2}\frac{1}{|\nhat-\zhat|}\mathscr{L}_{n'}\{\mathcal{T}_{\xhat}R(\qq;\nhat,\cdot)\mathcal{T}_{\xhat}\mathcal{M}_{1}(\qq;\kappa,\nhat,\cdot)\mathcal{T}_{\xhat}\js(\cdot)\}(\zhat)ds(\zhat)\\&+\;i\int_{\S^2}\mathscr{L}_{n'}\{\mathcal{T}_{\xhat}\mathcal{M}_{2}(\qq;\kappa,\nhat,\cdot)\mathcal{T}_{\xhat}\js(\cdot)\}(\zhat)ds(\zhat),\end{split}\end{equation*}
for some $n'=an+1$ with fixed $a>1$ and $n'-n>3$ (see  \cite[Theorem 3]{GaneshHawkins3}). By the use of \eqref{ingredient} and an additional identity for spherical harmonics \cite{ColtonKress} we obtain
\begin{equation}\label{discrete}
\begin{split}\mathcal{M}_{\kappa,n'}\,\js(\xhat)
=&\sum_{\rho'=0}^{2n'+1}\sum_{\tau'=1}^{n'+1}\mu_{\rho'}\nu_{\tau'}\alpha_{\tau'}\mathcal{T}_{\xhat}R(\qq;\nhat,\zhat_{\rho'\tau'})\mathcal{T}_{\xhat}\mathcal{M}_{1}(\qq;\kappa,\nhat,\zhat_{\rho'\tau'})\mathcal{T}_{\xhat}\js(\zhat_{\rho'\tau'})\\
&+\;i\sum_{\rho'=0}^{2n'+1}\sum_{\tau'=1}^{n'+1}\mu_{\rho'}\nu_{\tau'}\mathcal{T}_{\xhat}\mathcal{M}_{2}(\qq;\kappa,\nhat,\zhat_{\rho'\tau'})\mathcal{T}_{\xhat}\js(\zhat_{\rho'\tau'}),
\end{split}\end{equation}
where $\alpha_{\tau'}=\sum_{l=0}^{n'}P_{l}^{0}(\zeta_{\tau'}).$
We proceed in the same way for the operator $(\kappae\mathcal{C}_{\kappae}-\kappai\mathcal{C}_{\kappai})$ in order to obtain a first approximation denoted by $(\kappae\mathcal{C}_{\kappae}-\kappai\mathcal{C}_{\kappai})_{n'}$. 

In a second step, the boundary integral equations  are projected onto 
the space $\H_{n}$ by applying $\L_{n}$. Finally, the  systems 
\eqref{inteq1} and \eqref{inteq2} are discretized into 
$4\times\big((n+1)^2-1\big)$ equations for the 
$4\times\big((n+1)^2-1\big)$ unknown coefficients by applying 
the scalar product   $(\;\cdot\;| \Y^{(1)}_{lj})_{n}$ and 
$(\;\cdot\;| \Y^{(2)}_{lj})_{n}$, for $l=0,\hdots,n$ and 
$j=-l,\hdots,l$ to each equation.
\section{Numerical implementation and examples}
In this section we discuss the implementation of the numerical scheme described above and present some results to show the  accuracy of the method.

 In view of \eqref{discrete},  we need the following formula for $k=1,2$
$$\mathcal{T}_{\xhat_{\rho\tau}}\Y_{l,j}^{(k)}(\yhat_{\rho'\tau'})=\sum_{\tilde{j}=-l}^l F_{\tau l\tilde{j}j} e^{i(j-\tilde{j})\phi_{\rho}}\,T_{\xhat_{\rho\tau}}^{-1}\Y_{l,\tilde{j}}^{(k)}(\yhat_{\rho'\tau'})$$ with
$
F_{\tau l\tilde{j}j}=e^{i(j-\tilde{j})\frac{\pi}{2}}\sum_{|\tilde{l}|\leq l}d_{\tilde{j}\tilde{l}}^{l}\left(\frac{\pi}{2}\right)d_{j\tilde{l}}^{l}\left(\frac{\pi}{2}\right)e^{i\tilde{l}\theta_{\tau}}$ and
$d_{\tilde{j}\tilde{l}}^{l}\left(\frac{\pi}{2}\right)=2^j\sqrt{\frac{(l+|j|!)}{(l+|\tilde{l}|!)}\frac{(l-|j|!)}{(l-|\tilde{l}|!)}}\mathscr{P}_{l+j}^{\tilde{l}-j,-\tilde{l}-j}(0)
$.
Here $\mathscr{P}_{n}^{a,b}$ for $a,b\geq 0$ 
is the normalized Jacobi polynomial evaluated at zero given by 
$\mathscr{P}_{n}^{a,b}(0)=2^{-n}\sum_{t=0}^{n}(-1)^{t}
\smat{n+a\\n-t}
\smat{n+b\\t}.$ 
For $\tilde{l}-\tilde{j}<0$ or $-\tilde{l}-\tilde{j}<0$, we can compute $d_{\tilde{j}\tilde{l}}^{l}(\frac{\pi}{2})$ using the symmetry relations
$$d_{\tilde{j}\tilde{l}}^{l}(\beta)=(-1)^{\tilde{j}-\tilde{l}}d_{\tilde{l}\tilde{j}}^{l}(\beta)=d_{-\tilde{l}-\tilde{j}}^{l}(\beta)=d_{\tilde{l}\tilde{j}}^{l}(-\beta).$$
The discrete approximation of the  operator $\mathbf{M}_{\kappa}$ of   $\mathcal{M}_{\kappa}$ is of the form
$$\mathbf{M}_{\kappa}=\begin{pmatrix}\mathbf{M}_{1,1}&\mathbf{M}_{1,2}\\\mathbf{M}_{2,1}&\mathbf{M}_{2,2}\end{pmatrix},$$
where $\mathbf{M}_{a,b}$, for $a,b=1,2$ is a $\big((n+1)^2-1\big)\times\big((n+1)^2-1\big)$ matrix. 
The coefficients of $\mathbf{M}_{a,b}$,  for $1\leq l,l'\leq n$, $|j|\leq l$ and $|j'|\leq l'$ are given by $$\mathbf{M}^{ljl'j'}_{a,b}=(\L_{n}\mathcal{M}_{\kappa,n'}\Y^{(a)}_{l,j},\Y^{(b)}_{l'j'})_{n}.$$
 We denote $\yhat_{\rho'\tau'}^{\rho\tau}=T^{-1}_{\xhat_{\rho\tau}}(\yhat_{\rho'\tau'})$. The coefficient $\mathbf{M}^{ljl'j'}_{a,b}$ is computed via the following procedure:
 
 (i) \underline{step 1}: we set $\ee_{\theta_{\tau'}}^{\rho\tau} = T_{\xhat_{\rho\tau}}^{-1} \ee_{\theta_{\tau'}}$ and $\ee_{\phi_{\rho'}}^{\rho\tau} = T_{\xhat_{\rho\tau}}^{-1} \ee_{\phi_{\rho'}}$ and if $a=b=1$ we compute for $\tau=1,\hdots n+1$,  $\rho=0,\hdots 2n+1$, $\tau'=1,\hdots,n'+1$ and $\tilde{j}=-n,\hdots n$:
\begin{align*}
&E_{1}^{\rho\tau\tau'\tilde{j}}=\sum_{\rho'=1}^{2n'+1}\mu_{\rho'}\Big(\ee_{\theta_{\tau}}\cdot\boldsymbol{\mathcal{A}}_{k,n'}^{\rho\tau\rho'\tau'}\ee_{\theta_{\tau'}}^{\rho\tau}\Big)e^{i\tilde{j}\phi_{\rho'}},
&&E_{2}^{\rho\tau\tau'\tilde{j}}=\sum_{\rho'=1}^{2n'+1}\mu_{\rho'}\Big(\ee_{\theta_{\tau}}\cdot\boldsymbol{\mathcal{A}}_{k,n'}^{\rho\tau\rho'\tau'}\ee_{\theta_{\tau'}}^{\rho\tau}\Big)
\frac{i\tilde{j}e^{i\tilde{j}\phi_{\rho'}}}{\sin\theta_{\tau'}},\\
&E_{3}^{\rho\tau\tau'\tilde{j}}=\sum_{\rho'=1}^{2n'+1}\mu_{\rho'}\Big(\ee_{\phi_{\rho}}\cdot\boldsymbol{\mathcal{A}}_{k,n'}^{\rho\tau\rho'\tau'}\ee_{\phi_{\rho'}}^{\rho\tau}\Big)\frac{e^{i\tilde{j}\phi_{\rho'}}}{\sin\theta_{\tau}},
&&E_{4}^{\rho\tau\tau'\tilde{j}}=\sum_{\rho'=1}^{2n'+1}\mu_{\rho'}\Big(\ee_{\phi_{\rho}}\cdot\boldsymbol{\mathcal{A}}_{k,n'}^{\rho\tau\rho'\tau'}\ee_{\phi_{\rho'}}^{\rho\tau}\Big)
\frac{i\tilde{j}e^{i\tilde{j}\phi_{\rho'}}}{\sin\theta_{\tau}\sin\theta_{\tau'}},
\end{align*}
where $\boldsymbol{\mathcal{A}}_{k,n'}^{\rho\tau\rho'\tau'}$ is the matrix given by$$\boldsymbol{\mathcal{A}}_{k,n'}^{\rho\tau\rho'\tau'}=\big[\alpha_{\tau'}^{n'}R(\qq;\xhat_{\rho\tau},\yhat_{\rho'\tau'}^{\rho\tau})\mathcal{M}_{1}(\xhat_{\rho\tau},\yhat_{\rho'\tau'}^{\rho\tau})+\mathcal{M}_{2}(\xhat_{\rho\tau},\yhat_{\rho'\tau'}^{\rho\tau})\big].$$
If $a=2$ we replace, here above, $\ee_{\theta_{\tau}}$ by $-\ee_{\phi_{\rho}}$ and $\ee_{\phi_{\rho}}$ by $\ee_{\theta_{\tau}}$ and when $b=2$ we replace $\ee^{\rho\tau}_{\theta_{\tau'}}$ by $-\ee^{\rho\tau}_{\phi_{\rho'}}$ and $\ee^{\rho\tau}_{\phi_{\rho'}}$ by $\ee^{\rho\tau}_{\theta_{\tau'}}$.\medskip

(ii) \underline{step 2}: we compute for $\tau=1,\hdots n+1$, $l=1,\hdots,n$,  $\rho=0,\hdots 2n+1$ and $\tilde{j}=-n,\hdots n$
\begin{align*}
D_{1}^{\rho\tau l\tilde{j}}&=\sum_{\tau'=1}^{n'+1}\nu_{\tau'}\gamma_{l}^{\tilde{j}}\left[\frac{\partial P_{l}^{|\tilde{j}|}(\cos(\theta_{\tau'}))}{\partial \theta_{\tau'}}E_{1}^{\rho\tau\tau'\tilde{j}}+P_{l}^{|\tilde{j}|}(\cos(\theta_{\tau'}))E_{2}^{\rho\tau\tau'\tilde{j}}\right]\,,\\
D_{2}^{\rho\tau l\tilde{j}}&=\sum_{\tau'=1}^{n'+1}\nu_{\tau'}\gamma_{l}^{\tilde{j}}\left[\frac{\partial P_{l}^{|\tilde{j}|}(\cos(\theta_{\tau'}))}{\partial \theta_{\tau'}}E_{3}^{\rho\tau\tau'\tilde{j}}+P_{l}^{|\tilde{j}|}(\cos(\theta_{\tau'}))E_{4}^{\rho\tau\tau'\tilde{j}}\right]
\end{align*}
where
$\gamma_{l}^{\tilde{j}}=(-1)^{\frac{|j|+j}{2}}\sqrt{\frac{2l+1}{4\pi l(l+1)}\frac{(l-|\tilde{j}|!)}{(l+|\tilde{j}|!)}}.$

(iii) \underline{step 3}: we compute for $\tau=1,\hdots n+1$, $l=1,\hdots,n$,  $\rho=0,\hdots 2n+1$, $j=-n,\hdots n$ and $j'=-n,\hdots n$
\begin{align*}
&C_{1}^{\rho\tau lj}=\sum_{|\tilde{j}|\leq l}F_{\tau l\tilde{j}j}D_{1}^{\rho\tau l\tilde{j}}e^{i(j-\tilde{j})\phi_{\rho}}\,, 
&& B_{1}^{\tau j'lj}=\sum_{\rho=0}^{2n+1}\mu_{\rho}C_{1}^{\rho\tau lj}e^{-ij'\phi_{\rho}}\,,\\
&C_{2}^{\rho\tau lj}=\sum_{|\tilde{j}|\leq l}F_{\tau l\tilde{j}j}D_{2}^{\rho\tau l\tilde{j}}e^{i(j-\tilde{j})\phi_{\rho}}\,,
&& B_{2}^{\tau j'lj}=\sum_{\rho=0}^{2n+1}\mu_{\rho}C_{2}^{\rho\tau lj}(-ij')e^{-ij'\phi_{\rho}}\,.
\end{align*}
(iv) \underline{step 4}: we compute for $l'=1,\hdots n$, $j'=-n,\hdots n$, $l=1,\hdots,n$  and $j=-n,\hdots n$
$$\mathbf{M}^{ljl'j'}_{a,b}=\sum_{\tau=1}^{n+1}\nu_{\tau}\gamma_{l'}^{j'}\left[\frac{\partial P_{l'}^{|j'|}(\cos\theta_{\tau})}{\partial \theta_{\tau}}B_{1}^{\tau j'lj}+P_{l'}^{|j'|}(\cos\theta_{\tau})B_{2}^{\tau j'lj}\right]\,.$$
We use the same procedure to implement the discrete approximation $(\kappae\mathbf{C}_{\kappae}-\kappai\mathbf{C}_{\kappai})$ of the operator $(\kappae\mathcal{C}_{\kappae}-\kappai\mathcal{C}_{\kappai})$.

The discrete approximations of the operators $K_{\rm DM}$ and $K_{\rm IM}$ in 
\eqref{eq:directmethod} and \eqref{eq:indirectmethod} are given by
\begin{align*}
\mathbf{K}_{\rm DM}&=\begin{pmatrix}\left(1+\frac{\mue\kappai^2}{\mui\kappae^2}\right)&0\\0&\left(1+\frac{\mui}{\mue}\right)\end{pmatrix}+\begin{pmatrix}\mathbf{M}_{\kappae}-\frac{\mue\kappai^2}{\mui\kappae^2}\mathbf{M}_{\kappai}&\frac{\mue}{\kappae^2}(\kappae\mathbf{C}_{\kappae}-\kappai\mathbf{C}_{\kappai})\\\frac{1}{\mue}(\kappae\mathbf{C}_{\kappae}-\kappai\mathbf{C}_{\kappai})&\mathbf{M}_{\kappae}-\frac{\mui}{\mue}\mathbf{M}_{\kappai}\end{pmatrix}\,,\\
\mathbf{K}_{\rm IM}&=\begin{pmatrix}
\left(1+\frac{\mue\kappai^2}{\mui\kappae^2}\right)&0\\0&\left(1+\frac{\mui}{\mue}\right)
\end{pmatrix}-\begin{pmatrix}
\mathbf{M}_{\kappae}-\frac{\mue\kappai^2}{\mui\kappae^2}\mathbf{M}_{\kappai}
&\frac{1}{\mue}(\kappae\mathbf{C}_{\kappae}-\kappai\mathbf{C}_{\kappai})\\
\frac{\mue}{\kappae^2}(\kappae\mathbf{C}_{\kappae}-\kappai\mathbf{C}_{\kappai})
&\mathbf{M}_{\kappae}-\frac{\mui}{\mue}\mathbf{M}_{\kappai}
\end{pmatrix}\,.
\end{align*}

The discrete approximation of the right-hand side $2\transposee{(\gg,\ff)}$ of one  the boundary integral equation 
system is the vector $2\transposee{(\mathbf{g},\mathbf{f})}= 
2\transposee{(\mathbf{g}_{1},\mathbf{g}_{2},\mathbf{f}_{1},\mathbf{f}_{2})}$  
whose  coefficients are given by
$$\mathbf{g}_{k}^{lj}=(\L_{n}\big(\Pq\gg\big)|\Y^{(k)}_{lj})_{n}, \quad \text{ and}\quad 
\mathbf{f}_{k}^{lj}=(\L_{n}\big(\Pq\ff\big)|\Y^{(k)}_{lj})_{n}$$
for $k=1,2$, $l=1,\hdots,n$ and $j=-l,\hdots,l$.

The parametrized form of the far field operator $\FFop$  is 
$\pFFop\smat{\js\\\ms}=\pFFop_{1}\,\js+\pFFop_{2}\,\ms$
with
\begin{align*}
\left(\pFFop_{1}\js\right)(\xhat)&=\frac{i\kappae}{4\pi}\int_{\S^2}e^{-i\kappae\xhat\cdot \qq(\yhat)}
\xhat\times\big[\big(\ee_{\theta}(\yhat)\cdot\js(\yhat)\big)\tt_{1}(\yhat) +
\big(\ee_{\phi}(\yhat)\cdot\js(\yhat)\big)\tt_{2}(\yhat)\big]ds(\yhat)
\,,\\
\left(\pFFop_{2}\ms\right)(\xhat)&=\frac{\mue}{4\pi}\int_{\S^2}e^{-i\kappae\xhat\cdot \qq(\yhat)}
\xhat\times\big[\big(\ee_{\theta}(\yhat)\cdot\ms(\yhat)\big)\tt_{1}(\yhat)
+\big(\ee_{\phi}(\yhat)\cdot\ms(\yhat)\big)\tt_{2}(\yhat)\big]\times\xhat\,ds(\yhat)
\,.
\end{align*}
The discrete approximation of $\pFFop$ evaluated at the $2(n_{\infty}+1)^2$ Gauss-quadrature points on the unit far sphere is
\[
\dFFop=\begin{pmatrix}\dFFop_{1}&\dFFop_{2}\end{pmatrix}
\qquad\text{with}\qquad 
\dFFop_{a}=\begin{pmatrix}\dFFop_{a,1}&\dFFop_{a,2}\end{pmatrix},\quad a\in\{1,2\}
\]
where $\dFFop_{a,b}$, for $a,b=1,2$  is a $6(n_{\infty}+1)^2\times((n+1)^2-1)$ matrix.
The coefficients of $\dFFop_{a,b}$,  for $a,b=1,2$, $1\leq l'\leq n$, $|j'|\leq l'$ and $\rho=0,\hdots,2n_{\infty}+1$ and $\tau=1,\hdots,n_{\infty}+1$ are given by
\begin{align*}
\dFFop^{\rho\tau l'j'}_{1,b}=&\left(\pFFop_{2}
\Y^{(b)}_{l'j'}\right)(\xhat_{\rho\tau})
=\dfrac{i\kappae}{4\pi}\sum\limits_{\rho'=0}^{2n+1}\sum\limits_{\tau'=1}^{n+1}
\mu_{\rho'}\nu_{\tau'}e^{-i\kappae\xhat_{\rho\tau}\cdot \qq(\yhat_{\rho'\tau'})}\\
&\xhat_{\rho\tau}\times\big[\big(\ee_{\theta_{\tau'}}\cdot\Y_{l'j'}^{(b)}(\yhat_{\rho'\tau'}) \big) \tt_{1}(\yhat_{\rho'\tau'})
+ \big(\ee_{\phi_{\rho'}}\cdot\Y_{l'j'}^{(b)}(\yhat_{\rho'\tau'})\big) \tt_{2}(\yhat_{\rho'\tau'})\big]
\,,\\
\dFFop^{\rho\tau l'j'}_{2,b}=&
\left(\pFFop_{1}\Y^{(b)}_{l'j'}\right)(\xhat_{\rho\tau})
= \dfrac{\mue}{4\pi}
\sum\limits_{\rho'=0}^{2n+1}\sum\limits_{\tau'=1}^{n+1}
\mu_{\rho'}\nu_{\tau'}e^{-i\kappae\xhat_{\rho\tau}\cdot \qq(\yhat_{\rho'\tau'})}\\
&\xhat_{\rho\tau}\times\big[\big(\ee_{\theta_{\tau'}}\cdot\Y_{l'j'}^{(b)}(\yhat_{\rho'\tau'}) \big)\tt_{1}(\yhat_{\rho'\tau'})
+\big(\ee_{\phi_{\rho'}}\cdot\Y_{l'j'}^{(b)}(\yhat_{\rho'\tau'}) \big)\tt_{2}(\yhat_{\rho'\tau'})\Big]\times\xhat_{\rho\tau}
\,.
\end{align*}

Table \ref{table2} exhibits fast convergence for the far field pattern $\EE^{\infty}$ for analytic dielectric boundaries with parametric representations given in Table \ref{table1}. 
In each of the following examples we take $\kappai=2\kappae$ and $\mui=2\mue$.
As a first test we compute the electric far field denoted, $\EE^{\infty}_{\text{ps}}$, created by an 
off-center point source  located inside the dielectric: 
$$\EE^{\rm inc}(\xb)=\grad\FS{\kappae}{\xb-\s}\times\pp, \quad \s\in\Omega\text{ and }\pp\in\S^2\,.
$$ 
In this case the total exterior wave has to vanish so that  the far field pattern of the scattered wave $\EE^{\rm s}$ is the opposite of the far field pattern of the incident wave:
$$\EE^{\infty}_{\text{exact}}(\xhat)=-\frac{i\kappae}{4\pi}e^{-i\kappae\xhat\cdot \s}\;(\xhat\times\pp).$$
We choose $\s=\transposee{\big(0,\tfrac{0.1}{\sqrt{2}},-\tfrac{0.1}{\sqrt{2}}\big)}$ and 
$\pp=\transposee{(1,0,0)}$. In Table \ref{table2} we list the $L^{\infty}$ error (by taking the maximum of 
errors obtained over 1300 observed directions, i.e.\ $n_{\infty}=25$) 

 As a second example we compute the electric far field denoted, $\EE^{\infty}_{\text{pw}}$, created by the scattering of an incident plane wave: 
 $$\EE^{\rm inc}(\xb)=\pp\,e^{i\kappae \xb\cdot \dd},\quad \text{where }\dd,\pp\in \S^2\text{ and }\dd\cdot\pp=0.$$
In the tabulated results we show the real part and the imaginary part of the polarization component of the electric far field evaluated at the incident direction:
$[\EE^{\infty}_{\text{pl}}(\dd)]_{n}\cdot\pp$. We chose $\s=\transposee{\big(0,0,1\big)}$ and $\pp=\transposee{(1,0,0)}$.

\begin{table}[t]\footnotesize
\caption{\footnotesize Parametric representation of the dielectric interfaces \cite{GaneshGraham,Onaka}}
\centering
\begin{tabular}{c c}
\hline\hline
surface & parametric representation \\ [0.5ex]
\hline \\[-1.5ex]
peanut & $q\circ\psi(\theta,\phi)=r(\theta)\transposee{(\sin\theta\cos\phi,2\sin\theta\sin\phi,\cos\theta)}$, \\
&$r(\theta)=(1+\sqrt{2})^{-\frac{1}{2}}\big(\cos(2\theta)+\sqrt{1+\cos^2(2\theta)}\big)^{\frac{1}{2}}$; 
\vspace{2mm} \\ 
rounded tetrahedron & $q\circ\psi(\theta,\phi)=r(\theta,\phi)\psi(\theta,\phi)$, $r(\theta,\phi)=(H(\gamma,\gamma,\gamma)
+5^{-3}H(-\gamma,-\gamma,-\gamma))^{-\frac{1}{5}}$,\\
&$H(\gamma,\gamma,\gamma)=h(\gamma,\gamma,\gamma)^{p}+h(-\gamma,-\gamma,\gamma)^{p}+h(-\gamma,\gamma,-\gamma)^{p}+h(\gamma,-\gamma,-\gamma)^{p}$,\\&$h(a,b,c)=|\min(0,a\sin\theta\cos\phi+b\sin\theta\sin\phi+c\cos\theta)|$, $\gamma=\tfrac{1}{\sqrt{3}}$.
\\ [1ex]
\hline
\end{tabular}
\label{table1}
\end{table}

\begin{table}[ht]\footnotesize
\caption{\footnotesize Convergence of the forward solver for the dielectric scattering problem. The second column displays the error of the far field pattern for an interior 
point source, and the last two columns display point values of the far field pattern for 
a plane incident wave.}
\centering
\begin{tabular}{c c c c c}
\hline\hline \\[-2.5ex]
surface & $n$ & $||[\EE^{\infty}_{\text{ps}}]_{n}-\EE^{\infty}_{\text{exact}}||_{\infty}$ & $\Re[\EE^{\infty}_{\text{pw}}(\dd)]_{n}\cdot\pp$ & $\Im[\EE^{\infty}_{\text{pw}}(\dd)]_{n}\cdot\pp$ \\ [0.5ex]
\hline \\[-1.5ex]
peanut & 5 &  $2.0487{\rm E}-03$ & $0.932\,867\,728$ & $0.397\,947\,302$
\\$\kappae=\frac{\pi}{2}$, $\mue=1$& 10 & $4.2497{\rm E}-05$ & $0.928\,030\,045$ & $0.389\,280\,158$
\\& 15 & $2.5742{\rm E}-07$ & $0.928\,047\,903$ & $0.389\,255\,789$
\\& 20 & $1.9720{\rm E}-09$  & $0.928\,048\,382$ & $0.389\,255\,828$\vspace{2mm} \\ 
bean & 10 & $2.8156{\rm E}-03$  & $1.105\,343\,837$ & $0.638\,710\,973$ 
\\$R=1$&  15 & $3.0993{\rm E}-04$  & $1.105\,314\,530$ & $0.638\,637\,998$
\\$\kappae=\frac{\pi}{2}$, $\mue=1$& 20 & $4.3103{\rm E}-05$  & $1.105\,314\,829$ & $0.638\,637\,840$
\\ & 25 & $9.5110{\rm E}-06$ & $1.105\,314\,827$ & $0.638\,637\,840$\vspace{2mm} \\
rounded & 10 & $2.7042{\rm E}-04$  & $1.279\,476\,494$ & $0.386\,849\,086$
\\ tetrahedron& 15 & $2.7724{\rm E}-05$ & $1.287\,913\,912$ & $0.386\,574\,373$ 
\\$\kappae=\frac{\pi}{4}$, $\mue=1$& 20 & $4.8104{\rm E}-06$ & $1.287\,492\,784$ & $0.386\,693\,985$ 
\\ & 25 & $5.1661{\rm E}-07$&  $1.287\,532\,182$ & $0.386\,615\,646$
\\ [1ex]
\hline
\end{tabular}
\label{table2}
\end{table}

 \section{Operator formulations and IRGNM}
 
To make the operator formulation \eqref{IP} in the introduction precise, we first have to introduce
a set of admissible parametrizations $\mathcal{V}$ which form an open subset of a Hilbert 
space $\mathcal{X}$. As in the introduction, let $\F_k:\mathcal{V}\to \LL^2_{\tang}(\S^2)$, $k=1,\dots,m$ denote
the operator which maps a parametrization $\qq\in \mathcal{V}$ of a boundary $\Gamma$ to the far field 
pattern $\EE^{\infty}_k$ corresponding to the incident field $\EE^{\rm inc}_k$. These operators
may be combined into one operator $\F:\mathcal{V}\to \LL^2_{\tang}(\S^2)^m$, 
$\F(\qq):=(\F_1(\qq),\dots,\F_m(\qq))^\top$. 
We also combine the measured far field patterns into a vector
$\EE^{\infty}_\delta:=(\EE^{\infty}_{1,\delta},\dots, \EE^{\infty}_{m,\delta})^\top\in \LL^2_{\tang}(\S^2)^m$ 
such that the inverse problem can be written as
\begin{equation}\label{eq:IP2}
\F(\qq) = \EE^{\infty}_{\delta}\,.
\end{equation}
To compute an approximate solution to \eqref{eq:IP2} we use the Iteratively Regularized Gau{\ss}-Newton Method
(IRGNM).  To apply this method we show in section \ref{sec:Fder}
that the operator $\F$ is Fr\'echet differentiable and derive formulas to evaluate the Fr\'echet derivative $\F'[\qq]$ 
and its adjoint $\F'[\qq]^*$. Then the iterates of the IRGNM can be computed by
\begin{equation}\label{eq:IRGNM}
\qq_{N+1}^{\delta}:= \operatorname{argmin}\limits_{\qq\in \mathcal{X}} 
\left[\|\F(\qq_N^{\delta})+\F'[\qq_N^{\delta}](\qq-\qq_N^{\delta})-\EE^{\infty}_{\delta}\|_{\LL^2_{\tang}(\S^2)^m}^2
+ \alpha_{N}\|\qq-\qq_0\|^2\right]\,.
\end{equation}
Here $\qq_0=\qq_0^{\delta}$ is some initial guess (in our numerical experiments we always chose the unit sphere),
and the regularization parameters are chosen of the form $\alpha_{N}=\alpha_0 \paren{\frac{2}{3}}^N$. 
Since the objective functional in \eqref{eq:IRGNM} is quadratic and strictly convex, the first order
optimality conditions are necessary and sufficient, and the updates $(\partial\qq)_{N}:=\qq_{N+1}^{\delta}-\qq_N^{\delta}$
are the unique solutions to the linear equations
\begin{equation}\label{eq:IRGNMlineq}
\Big(\alpha_{N}\Id+\sum\limits_{k=1}^m\adjoint{\F_k'[\qq^\delta_{N}]}\F'_{k}[\qq^\delta_{N}]\Big)
(\partial\qq)_{N}^{\delta}=
\sum\limits_{k=1}^m\adjoint{\F'_{k}[\qq^\delta_{N}]}\big(\EE^{\infty}_{k,\delta}-\F_{k}(\qq^\delta_{N})\big)
+\alpha_{N}\big(\qq^{\delta}_{0}-\qq_{N}^{\delta}\big).
\end{equation}

\medskip
It remains to describe the choice of the set of admissible parametrizations $\mathcal{V}$ and the underlying
Hilbert space $\mathcal{X}$. A rather general way to parametrize a boundary $\Gamma$ is 
to choose some reference domain $\Omega_{\rm ref}$ with boundary $\Gammaref$  and consider mappings
$\qq:\Gammaref\to \Gamma$ belonging to
\begin{equation}\label{eq:gen_param}
\parset:=\left\{\qq\in H^s(\Gammaref,\R^3):\qq \mbox{ injective}, \det(D\qq(\xhat))\neq 0\mbox{ for all }\xhat\in\Gammaref\right\}\,.
\end{equation}
This will be convenient for describing Fr\'echet derivatives of $\F$ in section \ref{sec:Fder}. 
For $s>2$ the set $\parset$ is open in $\mathcal{X}:=  H^s(\Gammaref,\R^3)$ and $\mathcal{X}\subset  \mathscr{C}^1(\Gammaref,\R^3)$. 
If $\Gamma$ and $\Gammaref$ are sufficiently smooth, $\Gamma$ has a parametrization in $\parset$ 
if and only if $\Gamma$ and $\Gammaref$ have the same genus. 

A disadvantage of the choice \eqref{eq:gen_param} is that a given interface $\Gamma$ has many parametrizations in $\parset$. 
In the important special case that $\Gamma$ is star-shaped with respect to the origin, we can choose
$\Gammaref=\S^2$ and consider special parametrizations of the form 
\[
\qq= \mathcal{R}r   \qquad\mbox{with}\qquad
(\mathcal{R}r)(\xhat):=r(\xhat)\xhat,\quad\xhat\in\S^2
\] 
with a function $r:\S^2\to (0,\infty)$. Then the function $r$ is uniquely determined by $\Gamma$. 
In this case we choose the underlying Hilbert space $\mathcal{X}_{\rm star}:= H^s(\S^2,\R)$ with $s>2$ and
the set of admissible parametrizations by
\[
\parset_{\rm star}:=\{r\in\mathcal{X}_{\rm star}: r>0\}\,.
\]
As $\mathcal{R}(\parset_{\rm star})\subset \parset$,  we can define
$\F_{\rm star}:\parset_{\rm star}\to \LL^2_{\tang}(\S^2)^m$ by
\[
\F_{\rm star}:=\F\circ \mathcal{R}\,.
\]
Then $\F_{\rm star}$ is injective if a star-shaped interface $\Gamma$ is uniquely determined by  
the far field data $\EE^{\infty}_1,\dots,\EE^{\infty}_m$. 

\section{ The Fr\'echet derivative and its adjoint}\label{sec:Fder}


In this section we assume that the set $\parset$ of admissible parametrizations in chosen by \eqref{eq:gen_param} 
with some reference boundary $\Gammaref$.
For $\qq\in\parset$ we define $\Gamma_{\qq}:=\qq(\Gammaref)$ and denote  
by $\nn_{\qq}$ the exterior unit normal vector to $\Gamma_{\qq}$. More generally we will label all quantities 
and operators related to the dielectric scattering problem for the interface $\Gamma_{\qq}$ by the index $\qq$. 
We restrict our discussion to the case $m=1$ since the general case can be reduced to this special case
by the obvious formulas $\F'[\qq]\xi = (\F'_1[\qq]\xi,\dots,\F'_m[\qq]\xi)^\top$ and 
$\F'[\qq]^*\hh=\sum_{k=1}^m \F_k'[\qq]^*\hh$.

The following theorem was established in \cite{CostabelLeLouer2}. An alternative proof can be found in \cite{Hettlich2}.

\begin{theorem}[characterization of $\F'{[\qq]}$]\label{theo:Fprime}
The mapping $\F:\parset\rightarrow \LL^2_{\tang}(S^2)$ with $s>2$ 
is Fr\'echet differentiable at all $\qq\in\parset$ for which $\Gamma_{\qq}$ is of class $\mathscr{C}^2$, 
and  the  first derivative at $\qq$   in the direction $\ki\in\mathcal{X}$ 
is given by $$\F'[\qq]\ki=\EE^{\infty}_{\qq,\ki},$$
where $\EE^{\infty}_{\qq,\ki}$ is the far field pattern of the solution $(\EE^{\rm i}_{\qq,\ki},\EE^{\rm s}_{\qq,\ki})$ to the Maxwell equations \eqref{ME} in $\R^3\backslash\Gamma_{\qq}$ that satisfies the Silver-M\"uller radiation condition and the transmissions condition 
$$\left\{\begin{array}{c}\nn_{\qq}\times\EE^{\rm s}_{\qq,\ki}-\nn_{\qq}\times\EE^{\rm i}_{\qq,\ki}=\ff'_{\qq,\ki},\vspace{1mm}\\
\frac{1}{\mue}\nn_{\qq}\times\Rot\EE^{\rm s}_{\qq,\ki}-\frac{1}{\mui}\nn_{\qq}\times\Rot\EE^{\rm i}_{\qq,\ki}=\gg'_{\qq,\ki},\end{array}\right.$$
where
\begin{align*}
\ff'_{\qq,\ki}=&-\left(\ki\!\circ\!\qq^{-1}\!\cdot\nn_{\qq}\right)
\Big\{\nn_{\qq}\times\Rot(\EE_{\qq}^{\rm s}+\EE^{\rm inc})\times\nn_{\qq}
-\nn_{\qq}\times\Rot\EE_{\qq}^{\rm i}\times\nn_{\qq}\Big\}\\
&+\Rot_{\Gamma_{\qq}}\Big((\ki\!\circ\!\qq^{-1}\!\cdot\nn_{\qq})
\big(\nn_{\qq}\cdot(\EE_{\qq}^{\rm s}
+\EE^{\rm inc})-\nn_{\qq}\cdot\EE_{\qq}^{\rm i}\big)\Big),\\
\gg'_{\qq,\ki}=&-\left(\ki\!\circ\!\qq^{-1}\!\cdot\nn_{\qq}\right)
\Big\{\frac{\kappae^2}{\mue} \nn_{\qq}\times(\EE_{\qq}^{\rm s}+\EE^{\rm inc})\times\nn_{\qq}
-\frac{\kappai^2}{\mui}\big(\nn_{\qq}\times\EE^{\rm i}_{\qq}\big)\times\nn_{\qq}\Big\}\\
&+\Rot_{\Gamma_{\qq}}\left((\ki\!\circ\!\qq^{-1}\!\cdot\nn_{\qq})
\left\{\frac{1}{\mue}\nn_{\qq}\cdot\Rot(\EE_{\qq}^{\rm s}+\EE^{\rm inc})
-\frac{1}{\mui}\nn_{\qq}\cdot\Rot\EE^{\rm i}_{\qq}\right\}\right)
\end{align*}
on $\Gamma_{\qq}$ where $(\EE_{\qq}^{\rm i},\EE_{\qq}^{\rm s})$ is the solution of the dielectric scattering problem \eqref{ME}-\eqref{T3} with the interface $\Gamma_{\qq}$ 
and we have used Notation \ref{not:triple_cross}.
\end{theorem}

\begin{remark}[alternative form of boundary values]\label{remark}
{\rm  By straightforward calculations and the use of the transmission conditions, 
one can express the boundary values of the Fr\'echet derivative in terms of the 
solution to the system of integral equations \eqref{eq:directmethod} of the direct approach, i.e. 
\begin{equation}\label{eq:defi_u}
\begin{pmatrix}\uu^1_{\qq} \\ \uu^2_{\qq} \end{pmatrix}
= \uu^{\rm s}_{\qq}+\uu^{\rm inc} 
= \begin{pmatrix}
\nn_{\qq}\times(\EE_{\qq}^{\rm s}+\EE^{\rm inc})\\
\tfrac{1}{\mue}\nn_{\qq}\times\Rot(\EE_{\qq}^{\rm s}+\EE^{\rm inc})
\end{pmatrix}. 
\end{equation}
First we note that
$(\EE_{\qq}^{\rm s}+\EE^{\rm inc})=\frac{1}{\kappae^2}\Rot\Rot(\EE_{\qq}^{\rm s}+\EE^{\rm inc})\text{ and }\EE_{\qq}^{\rm i}=\frac{1}{\kappai^2}\Rot\Rot\EE^{\rm i}_{\qq}.$
Moreover, using the identity (see \eqref{eq:normal_curl}, \eqref{eq:n_cross_diffop}) 
\begin{equation}\label{eq:id_scurl_sdiv}
\nn_{\qq}\cdot \Rot \EE
= \rot_{\Gamma_{\qq}}(\nn_{\qq}\times \EE \times \nn_{\qq})
= - \Div_{\Gamma_{\qq}}(\nn_{\qq}\times \EE)\qquad \mbox{on }\Gamma_{\qq}\,,
\end{equation}
which holds for any smooth vector function $\EE$ defined on a neighborhood of $\Gamma_{\qq}$, we obtain 
\begin{subequations}\label{eqs:BC}
\begin{align}\label{BC1}
\ff'_{\qq,\ki}=&-\left(\ki\!\circ\!\qq^{-1}\!\cdot\nn_{\qq}\right)(\mue-\mui)
\uu^{(2)}_{\qq}\times\nn_{\qq}
-\left(\frac{\mue}{\kappae^2}-\frac{\mui}{\kappai^{2}}\right)
\Rot_{\Gamma_{\qq}}\left((\ki\!\circ\!\qq^{-1}\!\cdot\nn_{\qq})\Div_{\Gamma_{\qq}}
\uu^{(2)}_{\qq}\right)\\
\label{BC2}\gg'_{\qq,\ki}=&-\left(\ki\!\circ\!\qq^{-1}\!\cdot\nn_{\qq}\right)
\left(\frac{\kappae^2}{\mue}-\frac{\kappai^2}{\mui}\right)
\uu^{(1)}_{\qq}\times\nn_{\qq}
-\left(\frac{1}{\mue}-\frac{1}{\mui}\right)\Rot_{\Gamma_{\qq}}
\paren{(\ki\!\circ\!\qq^{-1}\!\cdot\nn_{\qq})\Div_{\Gamma_{\qq}}\uu^{(1)}_{\qq}}\,.
\end{align}
\end{subequations}
An interesting feature of these formulas is that they makes appear the contrasts between the interior and exterior values of the dielectric constants. 
}
\end{remark}

To define the adjoint of $F'[\qq]:\mathcal{X}=H^s(\Gammaref;\R^3)\to \LL^2_{\tang}(\S^2)$, 
we interpret the naturally complex Hilbert space $\LL^2_{\tang}(\S^2)$ as a real Hilbert space 
with the real-valued inner product $\Re \langle \cdot,\cdot\rangle_{\LL^2_{\tang}(\S^2)}$. 
For bounded linear operator between complex Hilbert spaces such a reinterpretation of the
spaces as real Hilbert spaces does not change the adjoint.

\begin{proposition}[characterization of the adjoint $\F'{[\qq]}^*$]\label{prop:adjoint}
Let 
$$\EE^{\rm inc}_{\hh}(\yb):=\frac{\mue}{4\pi}\int_{S^2}e^{-i\kappae\xhat\cdot \yb}\hh(\xhat)\;ds(\xhat),
\qquad \yb\in\R^3
$$ 
denote the vector Herglotz function with kernel $\hh\in\LL^2_{\tang}(\S^2)$ and 
$\EE_{\qq,\bar{\hh}}$ the total  wave solution to the scattering problem for the dielectric interface 
$\Gamma_{\qq}$ and the incident wave $\EE^{\rm inc}_{\bar{\hh}}$. Moreover, let $j_{\mathcal{X}\hookrightarrow \LL^2}$ 
denote the embedding operator from $\mathcal{X}= H^s(\Gammaref,\R^3)$ to $L^2(\Gammaref,\R^3)$. Then
\begin{align*}
\Adjoint{\F'[\qq]}\hh=j_{\mathcal{X}\hookrightarrow \LL^2}^*
\bigg(J_{\qq}\cdot\bigg(\nn_{\qq}\Re\bigg\{
&-(\mue-\mui)\Big(\frac{1}{\mue}\nn_{\qq}\times\Rot\overline{\EE_{\qq,\bar{\hh}}}\Big)
\cdot\overline{\uu^{(2)}_{\qq}}\\
&+\left(\frac{\mue}{\kappae^2}-\frac{\mui}{\kappai^{2}}\right)
\Div_{\Gamma_{\qq}}\Big(\frac{1}{\mue}\nn_{\qq}\times\Rot\overline{\EE_{\qq,\bar{\hh}}}\Big)
\cdot\Div_{\Gamma_{\qq}}\overline{\uu^{(2)}_{\qq}}\\
&-\left(\frac{\kappae^2}{\mue}-\frac{\kappai^2}{\mui}\right)
\Big(\nn_{\qq}\times\overline{\EE_{\qq,\bar{\hh}}}\Big)
\cdot\overline{\uu^{(1)}_{\qq}}\\
& +\left(\tfrac{1}{\mue}-\tfrac{1}{\mui}\right)
\Div_{\Gamma_{\qq}}\Big(\nn_{\qq}\times\overline{\EE_{\qq,\bar{\hh}}}\Big)
\cdot\Div_{\Gamma_{\qq}}\overline{\uu^{(1)}_{\qq}}
\bigg\}\bigg)\circ\qq\bigg)\;.
\end{align*}
\end{proposition}

\begin{proof} The proof consists of three steps:\\
\emph{1. factorization of $F'[\qq]$ and $F'[\qq]^*$:}
Due to Theorem \ref{theo:Fprime} and Remark \ref{remark} $F'[\qq]$ has a factorization 
\[\F'[\qq]\ki=A^{\qq}B^{\qq}\ki \qquad\mbox{where}\qquad
B^{\qq}\ki:=\begin{pmatrix}B_{1}^{\qq}\ki \\ B_{2}^{\qq}\ki
\end{pmatrix}
:=\begin{pmatrix}\gg'_{\qq,\ki}\\ \ff'_{\qq,\ki}\end{pmatrix}
\]
with $\ff'_{\qq,\ki}$ and $\gg'_{\qq,\ki}$ defined in \eqref{eqs:BC} and
$A^{\qq}$ maps the  boundary values $\smat{\gg_{\qq}\\ \fq }$
onto the far field pattern  of the transmission problem \eqref{ME}-\eqref{T1generique}-\eqref{T3} at the interface $\Gamma_{\qq}$, i.e.\ $A^{\qq}:=2\FFop^{\qq} (K_{\rm IM}^{\qq})^{-1}$. 

Let us denote by $(A^{\qq})_{\LL^2}^*$ and $(B^{\qq})_{\LL^2}^*$ the adjoints of $A^{\qq}$ and $B^{\qq}$ with
respect to the $L^2$ inner products. ($B^{\qq}$ is obviously unbounded and not everywhere defined 
from $L^2(\Gammaref,\R^3)$ to $\LL^2_{\tang}(\Gamma_{\qq})^2$, but well-defined on $H^1(\Gammaref,\R^3)$.
Moreover, $B^{\qq}(H^s(\Gammaref,\R^3))\subset \HH^{-1/2}_{\Div}(\Gamma)^2\cap \LL^2_{\tang}(\Gamma)^2$ 
for $s>2$.) Therefore, the adjoint of $F'[\qq]$ has the factorization
\[
F'[\qq]^*\hh = j_{\mathcal{X}\hookrightarrow \LL^2}^* (B^{\qq})_{\LL^2}^*(A^{\qq})_{\LL^2}^*\hh\,,
\] 
and it remains to characterize ($A^{\qq})_{\LL^2}^*$ and $(B^{\qq})_{\LL^2}^*$. \\
\emph{2. characterization of $(A^{\qq})_{\LL^2}^*$:} Let us introduce the operator 
$\FFop_0^{\qq}:\LL^2_{\tang}(\Gamma_{\qq})\to \LL^2_{\tang}(\S^2)$ by 
$\paren{\FFop_{0}^{\qq}\mm}(\xhat):=\frac{\mue}{4\pi}\xhat\times\int_{\Gamma}e^{-i\kappae\xhat\cdot \yb}
\mm(\yb)\,ds(\yb)\times\xhat\,.$ Then
\[
\paren{(\FFop_0^{\qq})^*\hh}(\yb) 
= \frac{\mue}{4\pi} \nn_{\qq}(\yb)\times \int_{\S^2}e^{i\kappae\xhat\cdot \yb} \hh(\xhat)\,ds(\xhat)
\times \nn_{\qq}(\yb)
= \nn_{\qq}(\yb)\times\overline{\EE^{\rm inc}_{\bar{\hh}}(\yb)}\times \nn_{\qq}(\yb)\,.
\]
As
$\big(\FFop^{\qq}\smat{\jj\\\mm}\big)(\xhat)=\frac{i\kappae}{\mue}\xhat\times (\FFop_{0}^{\qq}\jj)(\xhat)
+(\FFop_{0}^{\qq}\mm)(\xhat)$ 
we obtain
$$\paren{\FFop^{\qq}}_{\LL^2}^*\hh
=\begin{pmatrix}
\nn_{\qq}\times\frac{1}{\mue}\Rot\overline{\EE^{\rm inc}_{\overline{\hh}}}\times\nn_{\qq} \\
\nn_{\qq}\times\overline{\EE^{\rm inc}_{\overline{\hh}}}\times\nn_{\qq}
\end{pmatrix}\,.
$$
Therefore, using Remark \ref{rem:transposed} to pass from 
$K_{\rm IM}$ to $K_{\rm DM}$, it follows that 
\begin{align*}
2\overline{(A^{\qq})_{\LL^2}^*\hh}
&= 2\overline{((K_{\rm IM}^{\qq})^{-1})^*(\FFop^{\qq})_{\LL^2}^*\hh}
= 2((K_{\rm IM}^{\qq})^\top)^{-1} \overline{(\FFop^{\qq})_{\LL^2}^*\hh} \\
&= \begin{pmatrix}\nn_{\qq}\times & 0 \\ 0 &\nn_{\qq}\times\end{pmatrix}
(K_{\rm DM}^{\qq})^{-1}
\begin{pmatrix} 2\EE^{\rm inc}_{\overline{\hh}}\times \nn_{\qq} \\
\frac{2}{\mue} \Rot\EE^{\rm inc}_{\overline{\hh}}\times \nn_{\qq}
\end{pmatrix} 
=\begin{pmatrix}\frac{1}{\mue}\nn_{\qq}\times\Rot\EE_{\qq,\bar{\hh}}\times\nn_{\qq}\vspace{2mm}\\ \nn_{\qq}\times\EE_{\qq,\bar{\hh}}\times\nn_{\qq}\end{pmatrix}
\end{align*}
where we have used \eqref{inteq1} in the last line. \\
\emph{3. characterization of $(B^{\qq})_{\LL^2}^*$:}
It will be convenient to compute
$(B^{\qq})_{\LL^2}^*\smat{\gg_1\times\nn_{\qq}\\ \gg_2\times\nn_{\qq}}
=(B_{1}^{\qq})_{\LL^2}^*(\gg_1\times\nn_{\qq})+(B_{2}^{\qq})_{\LL^2}^*(\gg_2\times\nn_{\qq}).$
For $B_{1}^{\qq}$ using the integration by part formula \eqref{dualrot}  we obtain
\begin{align*}
\Re\left\langle \gg_1\times\nn_{\qq},B_1^{\qq}\ki\right\rangle_{\LL^2_{\tang}(\Gamma_{\qq})}
= \int_{\Gamma_{\qq}}(\ki\!\circ\!\qq^{-1}\!\cdot \nn_{\qq})\Re\Big\{&-
\paren{\tfrac{\kappae^2}{\mue}-\tfrac{\kappai^2}{\mui}} 
\Big(\gg_1\times\nn_{\qq}\Big)\cdot\Big(\overline{\uu^{(1)}_{\qq}}\times\nn_{\qq}\Big)\\
&- \left(\tfrac{1}{\mue}-\tfrac{1}{\mui}\right)
\rot_{\Gamma_{\qq}}(\gg_1\times\nn_{\qq})
\cdot\Div_{\Gamma_{\qq}} \overline{\uu^{(1)}_{\qq}}\Big\}\,ds\,.
\end{align*}
Together with the transformation formula $\int_{\Gamma_{\qq}} f\, ds 
= \int_{\Gamma_{\rm ref}} (f\circ q) J_q\, ds$ and the identities
$({\bf a}\times \nn)\cdot ({\bf b}\times \nn)
=(\nn\times {\bf a}\times \nn) \cdot {\bf b}$ 
and \eqref{eq:id_scurl_sdiv} this yields 
\begin{align*}
(B_1^{\qq})_{\LL^2}^*(\gg_1\times \nn_ {\qq})
= J_{\qq}\cdot\bigg(\nn_{\qq} \Re\Big\{&
-\paren{\tfrac{\kappae^2}{\mue}-\tfrac{\kappai^2}{\mui}} 
\Big(\nn_{\qq}\times\gg_1\times\nn_{\qq}\Big)\cdot \overline{\uu^{(1)}_{\qq}} \\
&+ \left(\tfrac{1}{\mue}\!-\!\tfrac{1}{\mui}\right)
\Div_{\Gamma_{\qq}}\gg_1\cdot\Div_{\Gamma_{\qq}}\overline{\uu^{(1)}_{\qq}}\Big\}\bigg)\circ \qq
\,.
\end{align*}
Together with the analogous formula for $(B_2^{\qq})_{\LL^2}^*\gg_2$ and parts 1 and 2 we obtain the assertion.
\end{proof}

\begin{remark}\label{rem:eval_rhs_der}
Recall from the transformation formalas \eqref{transformSurfaceDOP} that $(\Div_{\Gamma_{\qq}} {\bf v})\circ \qq 
= \frac{1}{J_{\qq}}\Div_{\S^2}(\Pq{\bf v})$ and $\Pq\Rot_{\Gamma_{\qq}}v = \Rot_{\S^2} (v\circ \qq)$. 
As both $\Div_{\S^2}$ and $\Rot_{\S^2}$ are diagonal with respect to the chosen bases of spherical
harmonics and vector spherical harmonics, the implementation of the formulas in Remark \ref{remark} and
Proposition \ref{prop:adjoint} is straightforward using our discretization.
\end{remark}

Using \cite[Corollay 4]{HohageHarbrecht} we obtain that 
$F_{\rm star}'[r]^*\hh = j_{\mathcal{X}_{\rm star}\hookrightarrow L^2} r^2 \Re\{\dots\}\circ \qq$
where the expression in the curly brackets coincides with that in Proposition 
\ref{prop:adjoint}.

\section{Implementation of the Newton method}
Let us summarize the numerical implementation of the $N$th regularized Newton step for the operator equation
$F(\qq)= \EE^{\infty}_{\delta}$ (see \ref{eq:IP2}):
\begin{enumerate}
\item\label{it:first_step} For the parametrization $\qq_N^{\delta}:\Gamma_{\rm ref}\to \R^3$ of the current reconstruction
$\Gamma_N^{\delta}:=\qq_N^{\delta}(\Gamma_{\rm ref})$ of the interface, evaluate the forward operator $F$ by solving
the discretized approximation $\mathbf{K}_{\rm DM}\mathbf{u}^{(k)} = 2\mathbf{u}^{{\rm inc},k}$ of the
integral equation \eqref{eq:directmethod} of the direct method for all incident waves $k=1,\dots,m$ using an LU decomposition
of the matrix $\mathbf{K}_{\rm DM}$. Save the Fourier coefficients of $\mathbf{u}^{(k)}$ of the total exterior fields
$(\nn\times (\EE^{\mathrm{s},k}+\EE^{\mathrm{inc},k}), \nn\times \Rot(\EE^{\mathrm{s},k}+\EE^{\mathrm{inc},k}))^{\top}$
on $\Gamma_N$. Finally compute the discrete far field patterns $\mathbf{E}^{\infty,k} = \mathbf{G}\mathbf{u}^{(k)}$ for the
$k$th incident wave and the interface $\Gamma_N^{\delta}$.  
\item\label{it:second_step} Now $F'[\qq_N^{\delta}]\ki$ can be evaluated for any $\ki$
by solving discretized versions $\mathbf{K}_{\rm IM} (\mathbf{m}^{(k)}, \mathbf{j}^{(k)})^{\top} = 
2(\mathbf{g}^{(k) '}_{\qq_N^{\delta},\ki},\mathbf{f}^{(k) '}_{\qq_N^{\delta},\ki})^{\top}$ of the
integral equation \eqref{eq:indirectmethod} for $k=1,\dots,m$. The right hand sides can easily be evaluated using the
solutions $\mathbf{u}^{(k)}$ from point \ref{it:first_step} (see Remark \ref{rem:eval_rhs_der}). For the inversion of the matrix 
$\mathbf{K}_{\rm IM}$ the LU-decomposition of $\mathbf{K}_{\rm DM}$ can be reused (see Remark \ref{rem:transposed}). 
Finally, $F'[\qq_N^{\delta}]\ki$ is approximated by the concatination of the vectors $\mathbf{G} (\mathbf{j}^{(k)}, \mathbf{m}^{(k)})^{\top}$
for $k=1,\dots, m$.\\ [0.5ex]
Similarly, to compute $F'[\qq_N^{\delta}]^*\hh$ with $\hh=(\hh^{(1)},\dots,\hh^{(m)})^{\top}$,
we compute traces of the total fields $\EE_{\qq,\overline{\hh^{(j)}}}$ for Herglotz incident
fields with kernels $\overline{\hh^{(j)}}$ by evaluating 
$2\mathbf{K}_{\rm IM}^{\top}\mathbf{G}^{\top} \overline{\mathbf{h}^{(j)}}$. 
Then we use the formula in 
Proposition \ref{prop:adjoint} and some up the results for each $j$ to obtain 
$F'[\qq_N^{\delta}]^*\hh$. 
\item Compute the next iterate $\qq_{N+1}^{\delta}$ by minimizing the quadratic Tikhonov functional \eqref{eq:IRGNM} 
(or solving the equivalent linear equation \eqref{eq:IRGNMlineq}) by the conjugate gradient method. 
In each CG step $F'[\qq_N^{\delta}]$ and $F'[\qq_N^{\delta}]^*$ are applied to some vectors as described in point \ref{it:second_step}.
\end{enumerate}

In the CG algorithm we only compute $L^2$ adjoint $F'[\qq_N^{\delta}]^*_{\LL^2}$ and 
evaluate norms in $\mathcal{X}=H^s(\S^2)$ using Proposition \ref{prop:sobo_S2} and norms
in $\mathcal{Y} = \LL^2_{\rm t}(\S^2)^m$ using a quadrature formula. 

Using the discrepancy principle the Newton iteration is stopped at the first index 
$N$ for which 
\[
\|F(\qq_N^{\delta})-\EE^{\infty}_{\delta}\|\leq \tau \delta
\]
where the constant is chosen as $\tau=4$. 

\appendix
\setcounter{equation}{0}  
\renewcommand{\thesection}{A} 

\section*{Appendix A: surface differential operators}\label{app:surface_diffop}
First we briefly recall the definitions and some properties of surface differential operators 
following \cite{Nedelec}. Assuming that  $\Gamma$ admits an atlas $(\Gamma_{i},\mathcal{O}_{i},\psi_{i})_{1\leq i\leq p}$, where $(\Gamma_{i})_{1\leq i\leq p}$ is a covering of open subset of $\Gamma$ and for $i=1,\hdots,p$, the function $\psi_{i}$ is a diffeomorphism (of class $\mathscr{C}^{1}$ at least) such that  $\psi_{i}^{-1}(\Gamma_{i})=\mathcal{O}_{i}\subset\R^2$, then when $\xb\in\Gamma_{i}$ we can write $\xb=\psi_{i}(\xi^x_{1},\xi^x_{2})$ where $(\xi^x_{1},\xi^x_{2})\in\mathcal{O}_{i}$. The tangent plane to $\Gamma$ at $\xb$ is generated by the vectors $$\ee_{1}(\xb)=\frac{\partial \psi_{i}}{\partial\xi_{1}}(\xi^\xb_{1},\xi^\xb_{2})\text{ and }\ee_{2}(\xb)=\frac{\partial \psi_{i}}{\partial\xi_{2}}(\xi^\xb_{1},\xi^\xb_{2}).$$ The unit outer normal vector to $\Gamma$ and the surface area element are given by$$\nn=\frac{\ee_{1}\times \ee_{2}}{|\ee_{1}\times \ee_{2}|}\quad \text{ and }\quad ds(\yb)=|\ee_{1}(\yb)\times \ee_{2}(\yb)|d\xi_{1}d\xi_{2}=J_{\psi_{i}}(\yb)\,d\xi_{1}d\xi_{2},$$ where $J_{\psi_{i}}$ denotes the determinant of the Jacobian matrix of  $\psi_{i}:\mathcal{O}_{i}\mapsto\Gamma_{i}$. The cotangent plane to $\Gamma$ at $\xb$ is generated by the vectors
$$\ee^1(\xb)=\frac{\ee_{2}(\xb)\times\nn(\xb)}{J_{\psi_{i}}(\xb)} \text{ and }\ee^2(\xb)=\frac{\nn(\xb)\times \ee_{1}(\xb)}{J_{\psi_{i}}(\xb)}.$$
For $i=1,2$, we have that $\ee_{i}\cdot \ee^j=\delta_{i}^j$ where $\delta_{i}^j$ represents the Kronecker symbol.

The tangential gradient and the tangential vector curl of any scalar function $u\in\mathscr{C}^1(\Gamma,\C)$ 
are defined for $\xb=\psi_{i}(\xi_{1}^\xb,\xi_{2}^\xb)\in\Gamma$ by 
\begin{align}\label{G}
\grad_{\Gamma}u(\xb)&=\frac{\partial (u\circ\psi_{i})}{\partial \xi_{1}}\circ\psi_{i}^{-1}(\xb)\,\ee^{1}(\xb)+\frac{\partial (u\circ\psi_{i})}{\partial \xi_{2}}\circ\psi_{i}^{-1}(\xb)\,\ee^2(\xb),\\
\label{RR}
\Rot_{\Gamma}u(\xb)&=\frac{1}{J_{\psi_{i}}(\xb)}\left(\frac{\partial (u\circ\psi_{i})}{\partial \xi_{2}}\circ\psi_{i}^{-1}(\xb)\,\ee_{1}(\xb)
-\frac{\partial (u\circ\psi_{i})}{\partial \xi_{1}}\circ\psi_{i}^{-1}(\xb)\,\ee_{2}(\xb)\right)\,.
\end{align}
such that $\grad_{\Gamma}u=(\grad \tilde{u})|_{\Gamma}$ and $\Rot_{\Gamma}u = \Rot(\tilde{u}\tilde{\nn})|_{\Gamma}$ 
for any smooth extension $\tilde{u}$ of $u$ to a neighborhood of $\Gamma$ and a smooth extension 
$\tilde{\nn}$ of $\nn$ as gradient of a distance function.
Moreover define the surface divergence of any vector function $\vv=v^1\ee_{1}+v^2\ee_{2}\in\mathscr{C}^1(\Gamma,\C^3)$ in the tangent plane to $\Gamma$ and the surface scalar curl of any vector function $\ww=w_{1}\ee^{1}+w_{2}\ee^{2}\in\mathscr{C}^1(\Gamma,\C^3)$ in the cotangent plane to $\Gamma$ 
or $\xb=\psi_{i}(\xi_{1}^\xb,\xi_{2}^\xb)\in\Gamma$  by 
\begin{align}\label{D}
\Div_{\Gamma}\vv(\xb)&=\frac{1}{J_{\psi_{i}}(\xb)}\left(\frac{\partial (J_{\psi_{i}}v^1)\circ\psi_{i}}{\partial\xi_{1}}+\frac{\partial (J_{\psi_{i}}v^2)\circ\psi_{i}}{\partial\xi_{2}}\right)\circ\psi_{i}^{-1}(\xb)\,,\\
\label{R}
\rot_{\Gamma}\ww(\xb)&=\frac{1}{J_{\psi_{i}}(\xb)}\left(\frac{\partial (w_{2}\circ\psi_{i})}{\partial\xi_{1}}-\frac{\partial (w_{1}\circ\psi_{i})}{\partial\xi_{2}}\right)\circ\psi_{i}^{-1}(\xb)\,.
\end{align}
These definitions are independent of the choice of the coordinate system, and the identities
\begin{align}
\label{eq:normal_curl}
&\nn\cdot(\Rot \EE)|_{\Gamma} = \rot_{\Gamma}(\nn\times \EE\times \nn)\,,\\
\label{eq:n_cross_diffop}
&\Rot_{\Gamma} u = (\grad_{\Gamma}u)\times \nn, && \rot_{\Gamma}(\ww) = \Div_{\Gamma}(\ww\times \nn)\,, \\
\label{eq:chain}
&\rot_{\Gamma}\grad_{\Gamma} u=0, &&\Div_{\Gamma}\Rot_{\Gamma} u=0 
\end{align}
hold true for $u$ and $\ww$ and any smooth vector function $\EE$ defined on a neighborhood of $\Gamma$.
By density arguments, the surface differential operators can be extended to  
Sobolev spaces. For $s\in\R$,  $\jj\in\HH_{\tang}^{s+1}(\Gamma)$ and $\varphi\in H^{-s}(\Gamma)$ we have the dualities
\begin{align}\label{dualgrad}
&\int_{\Gamma}(\Div_{\Gamma}\jj)\cdot\varphi\,ds=-\int_{\Gamma}\jj\cdot\grad_{\Gamma}\varphi\,ds\,,\\
\label{dualrot}
&\int_{\Gamma}(\rot_{\Gamma}\jj)\cdot\varphi\,ds=\int_{\Gamma}\jj\cdot\Rot_{\Gamma}\varphi\,ds\,.
\end{align}


\renewcommand{\thesection}{B} 
\section*{Appendix B: spherical harmonics and Sobolev spaces on $\S^2$}\label{app:spherical_harmonics}
In this appendix we recall the characterizations of Sobolev spaces on $\S^2$ by scalar and vector
spherical harmonics following \cite{Nedelec}. 
For $l\in\N$ and $0\leq j\leq l$, let $P_{l}^{j}$ denote the $j$-th associated Legendre function of order $l$.
Using the notation \eqref{eq:spherical_coo}, the spherical harmonics defined by 
$$Y_{l,j}(\xhat)=(-1)^{\frac{|j|+j}{2}}\sqrt{\frac{2l+1}{4\pi}\frac{(l-|j|!)}{(l+|j|!)}}P_{l}^{|j|}(\cos\theta)e^{ij\phi}$$for $j=-l,\hdots,l$ and $l=0,1,2,\hdots$. 

\begin{proposition}\label{prop:sobo_S2}
$\{Y_{l,j}:l,j\in \mathbb{Z},l\geq 0, |j|\leq l\}$
is a complete orthonormal system in $L^2(\S^2)$. The complex Hilbert spaces $H^s(\S^2)$ 
for $s\in\R$ can be characterized by 
\begin{equation*}\label{Hs}H^s(\S^2)=\left\{q=\sum_{l=0}^{\infty}\sum_{j=-l}^lc_{l,j}Y_{l,j};\;c_{l,j}\in \C \text{ and }\;\sum_{l=1}^{\infty}\sum_{j=-l}^l(1+l^2)^s|c_{l,j}|^2<+\infty\right\},\end{equation*}
with (equivalent) norm 
$||q||^2_{H^s}=\sum\limits_{l=1}^{\infty}\sum\limits_{j=-l}^l(1+l^2)^s|c_{l,j}|^2=\sum_{l=1}^{\infty}\sum_{j=-l}^l(1+l^2)^s\left|\int_{\S^2}q\cdot \overline{Y_{l,j}}ds\right|^2.$
A function $q\in H^s(\S^2,\C)$ is real valued if and only if $c_{l,j}=\overline{c_{l,-j}}$ 
for all  $l=0,1,\hdots$ and  $j=-l,\hdots,l$.
\end{proposition}

The tangential gradient of the spherical harmonics  is given by 
\begin{align*}
\grad_{\S^2}Y_{l,j}(\xhat)=\begin{cases}
(-1)^{\frac{|j|+j}{2}}\!\!\sqrt{\frac{2l+1}{4\pi}\frac{(l-|j|!)}{(l+|j|!)}}
\left(\!\!\frac{\partial P_{l}^{|j|}(\cos\theta)}{\partial \theta}e^{ij\phi}\ee_{\theta}
\!+\!ij\frac{P_{l}^{|j|}(\cos\theta)}{\sin\theta}e^{ij\phi}\ee_{\phi}\!\!\right),
& \sin\theta\not=0,\\
\sqrt{l(l+1)\frac{2l+1}{4\pi}}\left(\frac{(\cos\theta)^l}{2}\ee_{\theta}+ij\frac{(\cos\theta)^{l+1}}{2}\ee_{\phi}\right),&\sin\theta\!=\!0, |j|\!=\!1,\\
\transposee{(0,0,0)},&\sin\theta\!=\!0, |j|\!\not=\!1
\end{cases}
\end{align*}
for $l\in\N^*$ and $j\in\N$ with $|j|\leq l$ with
\begin{align*} 
\frac{\partial P_{l}^{|j|}(\cos\theta)}{\partial \theta}
=\begin{cases}
-(l+|j|)(l-|j|+1)P_{l}^{|j|-1}(\cos\theta)-|j|\frac{\cos\theta}{\sin\theta}P_{l}^{|j|}(\cos\theta),&(|j|\not=0),\\
P_{l}^1(\cos\theta),&\text{ otherwise}.
\end{cases}
\end{align*}
The tangential vector spherical harmonics are defined by 
\[
\Y^{(1)}_{l,j}=\frac{1}{\sqrt{l(l+1)}}\grad_{\S^2}Y_{l,j} \text{ and }\Y_{l,j}^{(2)}=\frac{1}{\sqrt{l(l+1)}}\Rot_{\S^2}Y_{l,j}
\]
for $j=-l,\hdots,l$ and $l=1,2,\hdots$ 
\begin{proposition}
The vector spherical harmonics form a complete orthonormal system in $\LL^2_{\tang}(\S^2)$. 
The complex Hilbert spaces $\HH^{-1/2}_{\Div}(\S^2)$ can be chararcterized by
\begin{align*}
\HH^{-\frac{1}{2}}_{\Div}(\S^2)&=\displaystyle{\left\{\sum_{l=1}^{\infty}\sum_{j=-l}^l\alpha_{l,j}\Y^{(1)}_{l,j}+ \beta_{l,j}\Y^{(2)}_{l,j};
\sum_{l=1}^{\infty}\sum_{j=-l}^l(l(l+1))^{\frac{1}{2}}|\alpha_{l,j}|^2
+\frac{|\beta_{l,j}|^2}{(l(l+1))^{\frac{1}{2}}}<\infty\right\}}
\end{align*}
with (equivalent) norm 
$||\jj||^2_{H^{-\frac{1}{2}}_{\Div}}=\sum\limits_{l=1}^{\infty}\sum\limits_{j=-l}^l(l(l+1))^{\frac{1}{2}}\left|\int_{\S^2}\jj\cdot \overline{\Y^{(1)}_{l,j}}ds\right|^2+(l(l+1))^{-\frac{1}{2}}\left|\int_{\S^2}\jj\cdot \overline{\Y^{(2)}_{l,j}}ds\right|^2.$
\end{proposition}

\paragraph*{Acknowledgment:}  We would like to thank Olha Ivanyshyn and Rainer Kress for matlab code  and for helpful discussions. 
 \small
 \bibliographystyle{siam}

\bibliography{iterativemethod}
   \end{document}